\documentclass[12pt,twoside]{amsart}

\usepackage{times,a4wide,amssymb,amsmath,amsthm,bbm,graphicx}
\usepackage[draft=false]{hyperref}
\usepackage[alphabetic]{amsrefs}

\theoremstyle{plain}
\newtheorem{thm}{Theorem}

\newtheorem{lemma}[thm]{Lemma}

\theoremstyle{remark}
\newtheorem*{rem}{Remark}
\theoremstyle{definition}
\newtheorem{defn}{Definition}

\newcommand{\C}{{\mathbb C}}
\newcommand{\R}{{\mathbb R}}
\newcommand{\E}{{\mathbb E}}
\newcommand{\Pro}{{\mathbb P}}
\newcommand{\ind}{{\mathbbm{1}}}

\DeclareMathOperator{\im}{Im}
\DeclareMathOperator{\dl}{Li_2}

\newcommand{\cN}{\mathcal{N}}
\newcommand{\cL}{\mathcal{L}}
\newcommand{\cP}{\mathcal{P}}
\newcommand{\cD}{\mathcal{D}}
\newcommand{\cC}{\mathcal{C}}
\newcommand{\cH}{\mathcal{H}}
\newcommand{\fp}{\mathfrak{P}}

\newcommand\inprod[2]{\ensuremath{\langle\!\langle#1,#2\rangle\!\rangle}}
\newcommand{\Abs}[1]{\left|#1\right|}
\newcommand{\Wi}[1]{:\mathrel{#1}:}
\newcommand{\Wiaa}[2]{\ensuremath{:\mathrel{|#1|^{2#2}}:}}
\newcommand{\eps}{\varepsilon}

\DeclareMathOperator{\var}{Var}
\DeclareMathOperator{\cov}{Cov}

\newenvironment{rlist}
{

\begin{enumerate}}
{\end{enumerate}}

\title[The increment of the argument for the GEF.]{Fluctuations of the increment of the argument for the Gaussian entire function.}
\thanks{The first author is supported by ISF Grants~~1048/11 and~166/11, by ERC Grant~335141 and by the Raymond and Beverly Sackler Post-Doctoral Scholarship 2013--14. The second author is supported by ISF Grants~166/11 and~382/15.}
\author{Jeremiah Buckley}
\address{Department of Mathematics, King's College London, Strand, London, WC2R 2LS, UK}
\email{jeremiah.buckley@kcl.ac.uk}
\author{Mikhail Sodin}
\address{School of Mathematical Sciences, Tel Aviv University, Tel Aviv 69978, Israel}
\email{sodin@post.tau.ac.il }
\begin{document}
\begin{abstract}
The Gaussian entire function is a random entire function, characterised by a certain invariance with respect to isometries of the plane. We study the fluctuations of the increment of the argument of the Gaussian entire function along planar curves. We introduce an inner product on finite formal linear combinations of curves (with real coefficients), that we call the signed length, which describes the limiting covariance of the increment. We also establish asymptotic normality of fluctuations.
\end{abstract}
\maketitle

Let $(\zeta_n)_{n=0}^\infty$ be a sequence of iid standard complex Gaussian random variables (that is, each $\zeta_n$ has density $\tfrac1\pi e^{-|z|^2}$ with respect to the Lebesgue measure on the plane), and define the Gaussian entire function by
\begin{equation}\label{eq: def GEF}
  f(z)=\sum_{n=0}^\infty \zeta_n \frac{z^n}{\sqrt{n!}}.
\end{equation}
A remarkable feature of this random entire function is the invariance of the distribution of its zero set with respect to isometries of the plane. The invariance of the distribution of $f$ under rotations is obvious, by the invariance of the distribution of each $\zeta_n$. The translation invariance arises from the fact that, for any $w\in\C$, the Gaussian processes $f(z+w)$ and $e^{z\overline{w}+\tfrac12|w|^2}f(z)$ have the same distribution; this follows, for instance, by inspecting the covariances
\begin{align*}
  \E\left[ e^{z_1\overline{w}+\tfrac12|w|^2} f(z_1) e^{\overline{z_2}w+\tfrac12|w|^2} \overline{f(z_2)} \right] &= e^{z_1\overline{z_2}+z_1\overline{w}+\overline{z_2}w+|w|^2}\\
  &= e^{(z_1+w)\overline{(z_2+w)}} = \E\left[ f(z_1+w) \overline{f(z_2+w)} \right].
\end{align*}
Further, by Calabi's rigidity, $f$ is (essentially) the only Gaussian entire function whose zeroes satisfy such an invariance (see \cite{HKPV}*{Chapter 2} for details and further references).

Given a large parameter $R>0$, the function $\log f(Rz)$ gives rise to multi-valued fields with a high intensity of logarithmic branch points, which is somewhat reminiscent of chiral bosonic fields as described by Kang and Makarov \cite{KM}*{Lecture 12}. One way to understand asymptotic fluctuations of these fields as $R\to\infty$ is to study asymptotic fluctuations of the increment of the argument of $f(Rz)$ along a given curve, which will be our concern in this paper. Note that, by the argument principle, if the curve bounds a domain $G$ then this observable coincides with the number of zeroes of $f$ in $RG$ (the dilation of the set $G$), up to a factor $2\pi$ (and a sign change if the curve is negatively oriented with respect to the domain it bounds).

We begin with the following definition.
\begin{defn}
  In what follows a \emph{curve} $\Gamma$ is always a $\mathcal C^1$-smooth regular oriented simple curve in the plane, of finite length\footnote{By finite length we mean finite and positive, we do not consider a single point to be a regular curve.}. An \emph{$\R$-chain} is a finite formal sum $\Gamma = \sum_i a_i \Gamma_i$, where $\Gamma_i$ are curves and the coefficients $a_i$ are real numbers.
\end{defn}
Note that if the coefficients $a_i$ are integer valued, then we can assign an obvious geometric meaning to the formal sum $\Gamma = \sum_i a_i \Gamma_i$.
\begin{defn}
  Given a curve $\Gamma$ and $R>0$ we define $\Delta_{R}(\Gamma)$ to be the random variable given by the increment of the argument of $f(Rz)$ along $\Gamma$. Given an $\R$-chain $\Gamma = \sum_i a_i \Gamma_i$ we define $\Delta_{R}(\Gamma)=\sum_i a_i \Delta_{R}(\Gamma_i)$.
\end{defn}
In order for this definition to make sense, we need to see that almost surely $f$ does not vanish on a fixed curve. Note that the mean number of zeroes in a (measurable) subset of the plane is proportional to the Lebesgue measure of the set. Since the number of zeroes on a fixed curve is a non-negative random variable, whose mean is zero, the required conclusion follows. A quantitative version of this is given by \cite{NSV}*{Lemma 8}.

It is worth pointing out that the observable $\Delta_{R}(\Gamma)$ is invariant with respect to rotations but not with respect to translations. Indeed, since the Gaussian functions $f(z+w)$ and $e^{z\overline{w} + \tfrac12|w|^2} f(z)$ are equidistributed, the observable $\Delta_{R}(\Gamma+w)$ has the same distribution as $\Delta_{R}(\Gamma) + R^2 \im(\overline{w}\int_\Gamma dz)$. Note that the term $R^2 \im(\overline{w}\int_\Gamma dz)$ is not random, and that it vanishes whenever $\Gamma$ is a closed chain. This implies that $\Delta_{R}(\Gamma+w)$ and $\Delta_{R}(\Gamma)$ have the same fluctuations, and furthermore hints that the mean of the random variable $\Delta_{R}(\Gamma)$ should be
\begin{equation}\label{eq:mean}
  \E[ \Delta_{R}(\Gamma)] = R^2 \im\Big( \int_\Gamma \bar{z}\,dz \Big).
\end{equation}
This formula is not difficult to justify, see the beginning of Section~\ref{sec: mean and var}.

We are interested in studying the asymptotic fluctuations of the observable $\Delta_{R}(\Gamma)$, as $R\to\infty$.  In order to understand the limiting covariance of $\Delta_{R}(\Gamma_1)$ and $\Delta_{R}(\Gamma_2)$ we introduce an inner product on $\R$-chains\footnote{Strictly speaking, we introduce an inner product on equivalence classes of $\R$-chains, where we identify two chains if their difference is the zero chain. We shall ignore this issue throughout.}.
\begin{defn}
  Suppose that $\Gamma_1$ and $\Gamma_2$ are curves, whose unit normal vectors are denoted $\hat{n}_1$ and $\hat{n}_2$ respectively. We define the \emph{signed length} of their intersection to be
  \begin{equation*}
    \cL(\Gamma_1,\Gamma_2)=\int_{\C} \ind_{\Gamma_1} \ind_{\Gamma_2} \langle\hat{n}_1,\hat{n}_2\rangle\,d\cH^1
  \end{equation*}
  where $\ind_{\Gamma_1}$ and $\ind_{\Gamma_2}$ are the indicator functions of the supports\footnote{By the support of a curve $\Gamma$ we mean the set $\{\gamma(t):t\in I\}\subset\C$ for a parameterisation $\gamma:I\to\C$ of $\Gamma$.} of the curves $\Gamma_1$ and $\Gamma_2$ respectively, $\langle\cdot,\cdot\rangle$ is the inner product on $\C$ given by the standard inner product on $\R^2$ (we shall frequently identify $\C$ with $\R^2$ without further comment) and $\cH^1$ is the one-dimensional Hausdorff measure. More generally, given $\R$-chains $\Gamma_1 = \sum_i a_i \Gamma_{i,1}$ and $\Gamma_2 = \sum_j b_j \Gamma_{j,2}$ we define
  \begin{equation*}
    \cL(\Gamma_1,\Gamma_2)=\sum_{i,j}a_ib_j\cL(\Gamma_{i,1},\Gamma_{j,2}).
  \end{equation*}
\end{defn}

This definition needs several comments.
\begin{rlist}
  \item If $\gamma_k:I_k\to\R^2$, $k=1,2$, are unit speed parameterisations of the curves $\Gamma_1$ and $\Gamma_2$ then, if $\gamma_1(t_1)\in\mathrm{image}(\gamma_2)$, we define $t_1^*=\tau(t_1)\in I_2$ to be the unique value such that $\gamma_1(t_1)=\gamma_2(t_1^*)$. We then have
      \begin{equation}\label{eq: sign length param}
        \cL(\Gamma_1,\Gamma_2)=\int_{I_1}\int_{I_2} \ind_{\cD}(\gamma_1(t_1),\gamma_2(t_2)) \langle \gamma_1'(t_1), \gamma_2'(t_2) \rangle \,d\delta_{t_1^*}(t_2)dt_1
      \end{equation}
      where $\cD=\{(x,y) \in\R^2 \times\R^2\colon x=y\}$ and $\delta_{s}(t_2)$ is the point mass at $t_2=s$.
  \item Since we deal with $\mathcal C^1$-smooth regular curves, for most of the intersection points of $\Gamma_1$ and $\Gamma_2$ the angle between the curves is either $0$ or $\pi$; there are at most countably many points where this does not hold. This means that in \eqref{eq: sign length param} we can replace the term $\langle \gamma_1'(t_1), \gamma_2'(t_2) \rangle$ by $\alpha(\gamma_1'(t_1), \gamma_2'(t_2))$ where
      \begin{equation*}
        \alpha(x,y)=
        \begin{cases}
          +1&\text{ if }\langle x, y \rangle = |x||y|,\\
          -1&\text{ if }\langle x, y \rangle = -|x||y|,\\
          0&\text{ otherwise},\\
        \end{cases}
      \end{equation*}
      where $|\cdot|$ is the standard Euclidean norm on $\R^2$. In other words, $\cL(\Gamma_1,\Gamma_2)$ indeed measures the signed length of the intersection of the curves $\Gamma_1$ and $\Gamma_2$, see Figure~\ref{fig: sign length alpha}.
\begin{figure}
  \centering
  \includegraphics[width=0.5\columnwidth]{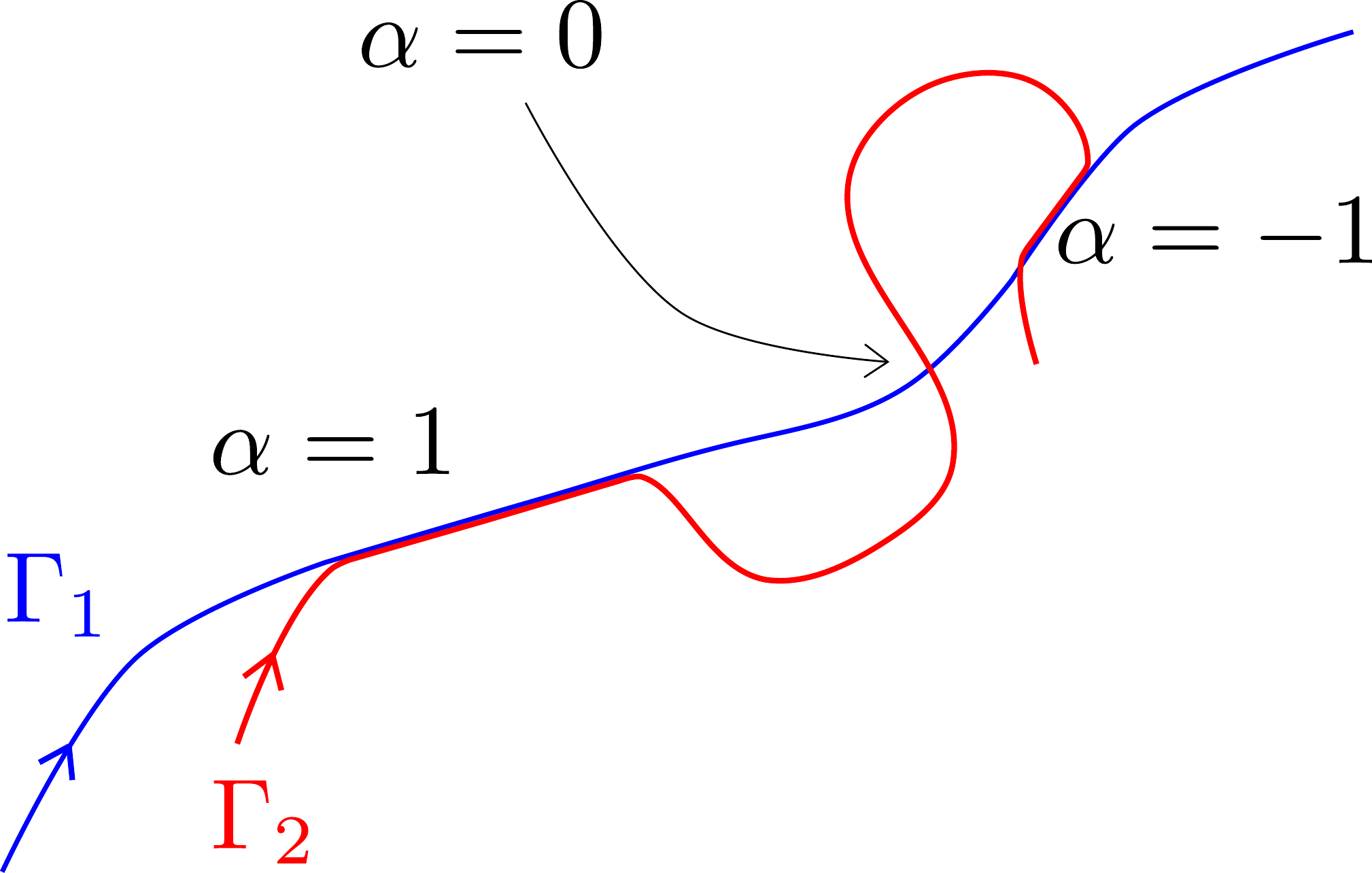}
  \caption{Illustration of the signed length of curves $\Gamma_1$ and $\Gamma_2$, the value of $\alpha(\gamma_1'(t_1), \gamma_2'(t_1^*))$ is indicated at the points of intersection}
  \label{fig: sign length alpha}
\end{figure}
  \item The signed length is a bilinear form on $\R$-chains, that is obviously symmetric. If $\Gamma=\sum_i a_i \Gamma_i$ then the associated quadratic form is
      \begin{equation*}
        \cL(\Gamma,\Gamma)=\int_{\R^2} \Big|\sum_i a_i \ind_{\Gamma_i} \hat{n}_i \Big|^2 \,d\cH^1.
      \end{equation*}
      We see that this quadratic form is non-negative and it vanishes if and only if $\Gamma$ is the zero chain, that is, $\sum_i a_i \ind_{\Gamma_i} \hat{n}_i$ is the zero function in $L^2(\cH^1)$. Thus the signed length defines an inner product on $\R$-chains.\footnote{It might be of some interest to describe the completion of this pre-Hilbert space, though for the purposes of this paper we shall have no need for such a description.}
\end{rlist}

We are ready to state our main result.
\begin{thm}\label{thm: main}
  Let $f$ be the Gaussian entire function \eqref{eq: def GEF}, and let $\Gamma = \sum_i a_i \Gamma_i$ be a non-zero $\R$-chain. Then, as $R\to\infty$,
  \begin{equation}\label{eq:var}
    \var \Delta_{R}(\Gamma)= \left(\frac{\sqrt{\pi}} 2\zeta\left(\frac32\right)+o(1)\right) \cL(\Gamma,\Gamma) R,
  \end{equation}
  where $\zeta$ is the Riemann zeta function, and the random variable
  \begin{equation*}
    \frac{\Delta_{R}(\Gamma)-\E[\Delta_{R}(\Gamma)]}{\sqrt{\var \Delta_{R}(\Gamma)}}
  \end{equation*}
  converges in distribution to the standard (real) Gaussian distribution.
\end{thm}

Less formally our result says that the observables $\Delta_{R}(\Gamma)$ have a scaling limit which is a Gaussian field built on the linear space of $\R$-chains equipped with the inner product defined by the signed length.

It is worth singling out a special case of Theorem~\ref{thm: main}, when each $\Gamma_i$ is the positively oriented boundary of a bounded domain $G_i$. In this case
\begin{equation*}
  \Delta_R(\Gamma) = 2\pi \sum_i a_i n_R(G_i)
\end{equation*}
where $n_R(G_i)$ is the number of zeroes of the entire function $f$ in the domain $RG_i$, the homothety of $G_i$ with scaling factor $R$. Here the Gaussian scaling limit is built on finite linear combinations $\sum_i a_i \ind_{G_i}$ and the limiting covariance of $n_R(G_i)$ and $n_R(G_j)$ is proportional to the signed length of $\partial G_i \cap \partial G_j$. Note that the same scaling limit appears in a physics paper of Lebowitz \cite{Leb} which deals with fluctuations of classical Coulomb systems.

The Gaussian scaling limit described in this special case corresponds to high-frequency fluctuations of linear statistics of the zero set of the Gaussian entire function $f$. For low frequencies the limiting Gaussian field is built on the Sobolev space $W_2^2$, which consists of $L^2$-functions whose weak Laplacian also belongs to $L^2$. This scaling limit was described in \cite{ST1}, see also \cite{NS}. The co-existence of different scaling limits of linear statistics, with different scaling exponents, is a curious feature of the zeroes of the Gaussian entire function. We expect that a similar phenomenon should arise in other natural homogeneous point processes with suppressed fluctuations (so-called superhomogeneous point processes).

Our work also has a one-dimensional analogue. The natural analogue of a curve in one dimension is the boundary of a finite interval and we attach a unit ``normal'' vector to each of the two end-points in the following manner: We say the interval is positively oriented if the normals are inward-pointing, that is, the normal on the left end-point points right, and the normal on the right end-point points left. Otherwise the interval is negatively oriented and the normals point in the opposite directions. Given two such boundaries $\partial I$ and $\partial J$, denoting the respective normals $\hat{n}_I$ and $\hat{n}_J$, we define an inner product by
\begin{equation*}
  \inprod{\partial I}{\partial J} = \frac12\int_{\R} \ind_{\partial I} \ind_{\partial J} \langle\hat{n}_I,\hat{n}_J\rangle\,d\cH^0
\end{equation*}
where $\cH^0$ is the (Hausdorff) counting measure, in analogy with the signed length (and we include the factor $\tfrac12$ to agree exactly with the results cited below). Given an ordered pair of distinct real numbers $(s,t)$, we identify the pair with the boundary of an interval which is positively oriented if $s<t$ and negatively oriented if $s>t$. The corresponding inner product is then
\begin{equation*}
  \inprod{(s,t)}{(s',t')} =
  \begin{cases}
    1&\text{ if } s=s'\text{ and }t=t'\\
    -1&\text{ if } s=t'\text{ and }t=s'\\
    \frac12&\text{ if } s=s'\text{ or }t=t'\text{ but not both}\\
    -\frac12&\text{ if } s=t'\text{ or }t=s'\text{ but not both}\\
    0&\text{ otherwise.}
  \end{cases}
\end{equation*}
This inner product appears as a limiting covariance in Gaussian limit theorems for eigenvalues of random unitary matrices \citelist{\cite{DE}*{Theorem 6.1} \cite{HKO}*{Theorem 2.2} \cite{W}*{Theorem 1}} and the logarithm of the Riemann zeta function on the critical line \cite{HNY}*{Theorem 1 and Section 2}.

We end this introduction with a brief discussion of the proof of Theorem~\ref{thm: main}. We follow the scheme developed in \cite{ST1}. The proof of the asymptotic \eqref{eq:var}, after some preliminaries, boils down to Laplace-type asymptotic evaluation of certain integrals. The proof of asymptotic normality uses the method of moments, and these moments are estimated using a combinatorial argument based on the diagram method. As often happens the devil is in the details: numerous difficulties\footnote{Note that somewhat similar difficulties were encountered by Montgomery in his study of discrepancies of uniformly distributed points \cite{Mon}*{Chapter 6, Theorem 3}.} arise from the fact that we cannot say much about the intersection of two ``nice'' curves other than that it is a one-dimensional compact subset of the plane. For example, if $\gamma_1(t)=t$ for $0\leq t\leq1$ and $\gamma_2(t)$ is an arbitrary $\cC^\infty$ $\C$-valued function on $[0,1]$, then the intersection of the corresponding curves can be an arbitrary closed subset of $[0,1]$. We also mention that it seems likely that one may apply the Fourth Moment Theorem of Peccati and Tudor \cite{PT4mom}*{Proposition 1} to see asymptotic normality, similar to \cite{MPRW}. We have not pursued this since, in our case, computing higher moments only introduces difficulties at the level of notation, and we do not think this a sufficient reason to employ such powerful machinery which relies on deep results from \cite{NP}.

Finally, a word on notation. We write $f\lesssim g$ to mean that $f\leq Cg$ for some constant $C$, which may depend on certain fixed parameters. If $f\lesssim g$ and $g\lesssim f$ then we write $f\simeq g$. We write $f=O(g)$ if $|f|\lesssim g$. We write $f=o(g)$ if $\tfrac fg\to0$ as $R\to\infty$. We write $f\sim g$ if $\tfrac fg\to1$ as $R\to\infty$.

\subsection*{Acknowledgements}
The authors thank Fedor Nazarov for a helpful discussion of the subtleties of the Laplace method, and Alexander Borichev and Nikolai Makarov for several useful conversations.

\section{Preliminary lemmas}
\subsection{Some elementary Gaussian estimates}
Suppose that $\zeta$ is a standard complex Gaussian random variable. Then a routine computation shows that for $p>-2$
\begin{equation}\label{eq: p mom 1 normal}
  \E[|\zeta|^p]=\Gamma(1+\tfrac p2),
\end{equation}
where $\Gamma$ is the Euler gamma function. An immediate consequence of \eqref{eq: p mom 1 normal} is the following.
\begin{lemma}\label{lem: poly of normal}
  Let $\zeta$ be a complex Gaussian random variable and let $Q$ be a polynomial. Then, for $1\leq p<+\infty$,
  \begin{equation*}
    \E\left[\left|Q\left(|\zeta|^2\right)\right|^p\right]<+\infty.
  \end{equation*}
\end{lemma}
The next lemma is also a simple consequence of \eqref{eq: p mom 1 normal}.
\begin{lemma}\label{lem: ratio normals}
  Let $\zeta_1$ and $\zeta_2$ be complex Gaussian random variables with $\E[|\zeta_2|^2]>0$, and let $1\leq p <2$. Then
  \begin{equation*}
    \E\left[\Abs{\frac{\zeta_1}{\zeta_2}}^p\right]<+\infty.
  \end{equation*}
\end{lemma}
\begin{proof}
  If $1<q< \tfrac2p$ and $q'$ is the H\"{o}lder conjugate of $q$ (i.e., $\tfrac1q + \tfrac1{q'}=1$), we have
  \begin{equation*}
    \E\left[\Abs{\frac{\zeta_1}{\zeta_2}}^p\right]\leq\E[|\zeta_1|^{pq'}]^{\tfrac 1{q'}}\E[|\zeta_2|^{-pq}]^{\tfrac 1{q}}<+\infty.\qedhere
  \end{equation*}
\end{proof}
The next lemma is given as an exercise in Kahane's celebrated book, for the reader's convenience we provide a proof.
\begin{lemma}[\cite{Kah}*{Chapter 12, Section 8, Exercise 3}]\label{lem: Kahane mean}
  Let $\zeta_1$ and $\zeta_2$ be jointly (complex) Gaussian random variables, with $\E[|\zeta_2|^2]\neq0$. Then
  \begin{equation*}
    \E\left[\frac{\zeta_1}{\zeta_2}\right]=\frac{\E[\zeta_1\overline{\zeta_2}]}{\E[|\zeta_2|^2]}.
  \end{equation*}
\end{lemma}
\begin{proof}
  Let $Z_1, Z_2$ be two i.i.d. $\cN_\C(0,1)$ random variables. Since $\zeta_1, \zeta_2$ are jointly Gaussian, there are $\alpha,\beta, \gamma \in \C$ such that the pair $(\zeta_1,\zeta_2)$ has the same distribution as $(\alpha Z_1+\beta Z_2,\gamma Z_1)$. In particular,
  \begin{equation*}
    \frac{\zeta_1}{\zeta_2} \overset{d}{=}  \frac{\alpha}{\gamma} + \frac{\beta}{\gamma} \frac {Z_2}{Z_1}.
  \end{equation*}
  Taking expectation, and recalling that $\E\left[\frac {Z_2}{Z_1} \right] = \E[ Z_2 ]\ \E\left[\frac 1{Z_1}\right] = 0$, we get
  \begin{equation*}
    \E\left[  \frac{\zeta_1}{\zeta_2} \right]=  \frac{\alpha}{\gamma}.
  \end{equation*}
  All that remains is to note that $\E[\zeta_1\overline{\zeta_2}]= \alpha \overline{\gamma}$ and that $\E[|\zeta_2|^2]=|\gamma|^2$.
\end{proof}
\begin{lemma}[\cite{Feld}*{Lemma B.2}]\label{lem: -pth mom 2 normals}
  Let $\zeta_1$ and $\zeta_2$ be $\mathcal{N}_\C(0,1)$ random variables with $\E[\zeta_1\bar{\zeta_2}]=\theta$ and suppose that $|\theta|\geq c>0$ and  $1\leq p<2$. Then
\begin{equation*}
  \E[|\zeta_1\zeta_2|^{-p}]\leq C(p,c) (1-|\theta|^2)^{1-p}.
\end{equation*}
\end{lemma}
\begin{rem}
  If $|\theta|=1$ then the expectation is divergent, even for $p=1$.
\end{rem}

\subsection{Gradients}
For convenience we write $f_R(z)=f(Rz)$, $K_R(z,w) = K(Rz,Rw)=e^{R^2z\overline{w}}$ and define
\begin{equation*}
  \widehat{f}_R(z)=\frac{f_R(z)}{\sqrt{K_R(z,z)}}
\end{equation*}
and note that $\widehat{f}_R(z)$ is a $\mathcal{N}_\C(0,1)$ random variable that satisfies
\begin{equation*}
  \widehat{K}_R(z,w)=\E[\widehat{f}_R(z)\overline{\widehat{f}_R(w)}]=\frac{K_R(z,w)}{\sqrt{K_R(z,z)K_R(w,w)}}.
\end{equation*}
Furthermore $|\widehat{K}_R(z,w)|=e^{-R^2|z-w|^2/2}$. To simplify our notation, we define 
\begin{equation*}
g_R(z)=\big|\widehat{f}_R(z)\big|^2=\Abs{f_R(z)}^2e^{-R^2|z|^2}.
\end{equation*}
The next lemma will be important later.
\begin{lemma}\label{lem: grad f^}
  Given a compact $K$ and $1\leq p <+\infty$, we have
  \begin{equation*}
    \E\big[ \big|\nabla g_R(z)\big|^p \big] \leq C(p,K,R)
  \end{equation*}
  for all $z\in K$.
\end{lemma}
\begin{proof}
  It is easy to see that
  \begin{equation*}
    \big|\nabla g_R(z)\big| \lesssim |f'_R(z)f_R(z)|e^{-R^2|z|^2} + R^2|z||\widehat{f}_R(z)|^2.
  \end{equation*}
  Trivially $\E\left[ |\widehat{f}_R(z)|^{2p} \right]$ is finite and independent of $z$, and Cauchy-Schwartz implies that
  \begin{equation*}
    \E\big[ \big|f'_R(z)f_R(z)e^{-R^2|z|^2}\big|^p \big] \leq \E\big[ \big|f'_R(z)e^{-R^2|z|^2/2}\big|^{2p} \big]^{1/2} \E\big[ \big|\widehat{f}_R(z)\big|^{2p} \big]^{1/2} \leq C(p,K,R),
  \end{equation*}
  since $f_R'(z)$ is a complex Gaussian with variance $(R^2+R^4|z|^2)e^{R^2|z|^2}$.
\end{proof}
\begin{lemma}\label{lem: grad poly}
  Given a compact $K$, a polynomial $Q$ and $1\leq p <+\infty$, we have
  \begin{equation*}
    \E\left[ \Abs{\nabla(Q\circ g_R)(z)}^p \right] \leq C(p,K,Q,R)
  \end{equation*}
  for all $z\in K$.
\end{lemma}
\begin{proof}
  Since $\nabla(Q\circ g_R)(z)= Q'(g_R(z)) \cdot \nabla g_R(z)$ this lemma follows from Cauchy-Schwarz, Lemma~\ref{lem: grad f^} and Lemma~\ref{lem: poly of normal}.
\end{proof}

\subsection{Interchange of operations}
In the proof of Theorem~\ref{thm: main} we will repeatedly need to apply Fubini's Theorem and exchange derivatives with expectation. In this subsection we prove some lemmas that will allow us to do precisely this. Throughout this section $\Gamma_1,\dots\Gamma_N$ will be curves and $\hat{n}_j$ will denote the normal vector to the curve $\Gamma_j$ at the point $z_j\in\Gamma_j$. We begin with a lemma that covers all of the cases we need.
\begin{lemma}\label{lem: abstract interchange}
  Let $\psi_j:\R_+\to\R$ be differentiable functions for $1\leq j\leq N$ and let $\Psi_j=\psi_j\circ g_R$. Suppose that
  \begin{equation}\label{eq: pth mom grad Lj}
    \int_{\prod_{j=1}^N\Gamma_j} \E\Big[ \Big| \prod_{j=1}^N \nabla\Psi_j(z_j) \Big| \Big] \prod_{j=1}^N |dz_j| <+\infty
  \end{equation}
  and that, for almost every tuple $(z_1,\dots,z_N)$ with respect to the measure $\prod_{j=1}^N |dz_j|$, there exists $\eps_0>0$ and $1<p<2$ such that
  \begin{equation}\label{eq: bdd grad p mom}
    \sup_{w_j\in D(z_j,\eps_0)}\E\Big[ \Big| \prod_{j=1}^N \nabla\Psi_j(w_j) \Big|^p \Big]<+\infty.
  \end{equation}
  Then
  \begin{align}\label{eq: interch}
    \E\bigg[ \int_{\prod_{j=1}^N\Gamma_j} \frac{\partial^N}{\partial \hat{n}_1\cdots\partial \hat{n}_N} \prod_{j=1}^N  \Psi_j(z_j) &\prod_{j=1}^N |dz_j| \bigg]\notag\\
    &=\int_{\prod_{j=1}^N\Gamma_j} \frac{\partial^N}{\partial \hat{n}_1\cdots\partial \hat{n}_N} \E\left[\prod_{j=1}^N \Psi_j(z_j) \right]\,\prod_{j=1}^N |dz_j|.
  \end{align}
\end{lemma}
\begin{rem}
  Trivially \eqref{eq: pth mom grad Lj} implies that the left-hand side of \eqref{eq: interch} is well defined. However, as will be clear from the proof, we can only infer that the integrand on the right-hand side, that is the term  $\frac{\partial^N}{\partial \hat{n}_1\cdots\partial \hat{n}_N} \E\left[\prod_{j=1}^N \Psi_j(z_j) \right]$, is well-defined at the points where \eqref{eq: bdd grad p mom} holds.
\end{rem}
\begin{proof}
  Note that \eqref{eq: pth mom grad Lj} immediately implies, by Fubini, that
  \begin{align*}
    \E\bigg[ \int_{\prod_{j=1}^N\Gamma_j} \frac{\partial^N}{\partial \hat{n}_1\cdots\partial \hat{n}_N} \prod_{j=1}^N  \Psi_j(z_j) &\prod_{j=1}^N |dz_j| \bigg]\\
    &=\int_{\prod_{j=1}^N\Gamma_j} \E\left[ \frac{\partial^N}{\partial \hat{n}_1\cdots\partial \hat{n}_N} \prod_{j=1}^N \Psi_j(z_j) \right]\,\prod_{j=1}^N |dz_j|.
  \end{align*}
  It therefore suffices to show that, for almost every tuple $(z_1,\dots,z_N)$ with respect to the measure $\prod_{j=1}^N |dz_j|$,
  \begin{equation}\label{eq: 1st suff cond}
    \E\left[ \frac{\partial^N}{\partial \hat{n}_1\cdots\partial \hat{n}_N} \prod_{j=1}^N \Psi_j(z_j) \right] = \frac{\partial^N}{\partial \hat{n}_1\cdots\partial \hat{n}_N} \E\left[ \prod_{j=1}^N \Psi_j(z_j) \right].
  \end{equation}

  Fix a tuple $(z_1,\dots,z_N)$ satisfying \eqref{eq: bdd grad p mom} for $\eps_0$ and $p$, and define, for $\eps_j<\eps_0$,
  \begin{equation*}
    h_j(\eps_j)=\frac{\Psi_j(z_j + \eps_j\hat{n}_j) - \Psi_j(z_j)}{\eps_j}.
  \end{equation*}
  We will show that
  \begin{equation}\label{eq:2nd suff cond}
    \lim_{\eps_1,\dots,\eps_N\to0} \E\Big[ \prod_{j=1}^N h_j(\eps_j) \Big] = \E\Big[ \lim_{\eps_1,\dots,\eps_N\to0} \prod_{j=1}^N h_j(\eps_j) \Big]
  \end{equation}
  which will imply \eqref{eq: 1st suff cond}, and therefore prove the lemma.

  We begin by establishing the existence of the inner limit on the right-hand side of \eqref{eq:2nd suff cond}. Notice first that, almost surely, $f_R$ does not vanish on the line intervals joining $z_j$ to $z_j+\eps_0\hat{n}_j$. Therefore, there exist some (random) neighbourhoods of these intervals where the gradient $\nabla\Psi_j$ is a well-defined function. We conclude that the limits
  \begin{equation*}
    \lim_{\eps_j\to0} h_j(\eps_j) = \bigg\langle \nabla\Psi_j(z_j) , \hat{n}_j \bigg\rangle = \frac{\partial}{\partial \hat{n}_j} \Psi_j(z_j)
  \end{equation*}
  exist almost surely. Finally we show that
  \begin{equation}\label{eq:3rd suff cond}
    \sup_{0<\eps_j<\eps_0} \E\Big[ \Big|\prod_{j=1}^N h_j(\eps_j)\Big|^p \Big]\leq C(p).
  \end{equation}
  By a standard argument, this implies that $\prod_{j=1}^N h_j(\eps_j)$ for $0<\eps_j<\eps_0$ is a uniformly integrable class of functions, and since we have already showed almost sure convergence (and therefore convergence in measure), we may infer \eqref{eq:2nd suff cond}.

  Once more we note that, almost surely, $f_R$ does not vanish on the line interval joining $z_j$ to $z_j+\eps_0\hat{n}_j$. This implies that
  \begin{equation*}
    |h_j(\eps_j)| \leq \frac1{\eps_j} \int_0^{\eps_j} \Abs{\nabla\Psi_j(z_j+t_j\hat{n}_i) } \,dt_j,
  \end{equation*}
  whence,
  \begin{equation*}
    |h_j(\eps_j)|^p \leq \frac1{\eps_j} \int_0^{\eps_j} \Abs{\nabla\Psi_j(z_j+t_j\hat{n}_i) }^p \,dt_j.
  \end{equation*}
  We get
  \begin{align*}
    \E\Big[ \Big| \prod_{j=1}^N h_j(\eps_j) \Big|^p \Big] & \leq \E\Big[ \frac1{\eps_1\dots\eps_N} \int_0^{\eps_1}\dots\int_0^{\eps_N} \Big| \prod_{j=1}^N \nabla\Psi_j(z_j+t_j\hat{n}_i) \Big|^p \prod_{j=1}^N dt_j \Big]\notag\\
    &= \frac1{\eps_1\dots\eps_N} \int_0^{\eps_1}\dots\int_0^{\eps_N} \E\left[ \Big| \prod_{j=1}^N \nabla\Psi_j(z_j+t_j\hat{n}_i) \Big|^p \right] \prod_{j=1}^N dt_j,
  \end{align*}
  by Fubini. By \eqref{eq: bdd grad p mom} we see that this is bounded uniformly in $\eps_1,\dots,\eps_N$, which is precisely \eqref{eq:3rd suff cond}.
\end{proof}

We now show that the hypothesis of this previous lemma hold in each of the specific cases we will need.
\begin{lemma}\label{lem: hypoth sat polys}
  Suppose that $\psi_j$ are polynomials for $1\leq j\leq N$. Then \eqref{eq: pth mom grad Lj} and \eqref{eq: bdd grad p mom} hold.
\end{lemma}
\begin{rem}
  In this case, \eqref{eq: bdd grad p mom} holds for every tuple $(z_1,\dots,z_N)$.
\end{rem}
\begin{proof}
  First note that, repeatedly applying Cauchy-Schwarz, both \eqref{eq: pth mom grad Lj} and \eqref{eq: bdd grad p mom} follow if we see that
  \begin{equation*}
    \E\left[ \Abs{\nabla\Psi_j(z_j) }^p \right]
  \end{equation*}
  is uniformly bounded for $z_j$ in a compact and any $p\geq1$. But this is precisely the conclusion of Lemma~\ref{lem: grad poly}.
\end{proof}
\begin{lemma}\label{lem: hypoth sat log poly}
  Suppose that $\psi_1=\log$ and $\psi_2$ is a polynomial. Then \eqref{eq: pth mom grad Lj} and \eqref{eq: bdd grad p mom} hold (with $N=2$).
\end{lemma}
\begin{rem}
  In this case, \eqref{eq: bdd grad p mom} holds for every pair $(z_1,z_2)$.
\end{rem}
\begin{proof}
  Suppose that $1\leq p<2$, choose $1<q< \tfrac 2p$ and let $q'$ be the H\"{o}lder conjugate of $q$ (i.e., $\tfrac1q + \tfrac1{q'}=1$). Note that
  \begin{align*}
    \E\Big[ \big| \nabla\Psi_1(z_1) \nabla\Psi_2(z_2) \big|^p \Big] \leq  \E\Big[ \big| \nabla\Psi_1(z_1) \big|^{pq} \Big]^{1/q} \E\Big[ \big| \nabla\Psi_2(z_2) \big|^{pq'} \Big]^{1/q'}.
  \end{align*}
  Once more, applying Lemma~\ref{lem: grad poly}, the term involving $\Psi_2$ is uniformly bounded. It therefore suffices to see that
  \begin{equation*}
    \E\Big[ \big| \nabla\Psi_1(z_1) \big|^{pq} \Big]
  \end{equation*}
  is uniformly bounded for $z_1$ in a compact and $1<pq<2$. Since
  \begin{equation*}
    \Psi_1(z_1)= \log |\widehat{f}_R(z_1)|^2 = \log |f_R(z_1)|^2 + R^2|z_1|^2
  \end{equation*}
  we have
  \begin{equation*}
    \big| \nabla\Psi_1(z_1) \big| \lesssim \Abs{\frac{f_R'(z_1)}{f_R(z_1)}} + R^2|z_1|,
  \end{equation*}
  and Lemma~\ref{lem: ratio normals} completes the proof.
\end{proof}
\begin{lemma}\label{lem: hypoth sat 2 logs}
  Suppose that $\psi_1=\psi_2=\log$. Then (with $N=2$) \eqref{eq: pth mom grad Lj} holds and for every pair $(z_1,z_2)$ with $z_1\neq z_2$, \eqref{eq: bdd grad p mom} holds.
\end{lemma}
\begin{proof}
  First fix $1\leq p<2$, let $1<q<\tfrac2p$ and let $q'$ be the H\"{o}lder conjugate of $q$. Then, for $w_1\neq w_2$,
  \begin{align*}
    \E\left[\Abs{\frac{f_R'(w_1)}{f_R(w_1)} \frac{f_R'(w_2)}{f_R(w_2)}}^p \right] &\leq \left(\E\Big[ \big|f_R(w_1) f_R(w_2) \big|^{-pq}\Big]\right)^{\tfrac1{q}} \left(\E\Big[ \big|f_R'(w_1) f_R'(w_2) \big|^{pq'} \Big] \right)^{\tfrac1{q'}}\notag\\
    &\leq \left(\E\Big[ \big| f_R(w_1) f_R(w_2)\big|^{-pq} \Big] \right)^{\tfrac1{q}} \left(\E\Big[ \big|f_R'(w_1) \big|^{2pq'} \Big] \E\Big[ \big|f_R'(w_2)\big|^{2pq'} \Big] \right)^{\tfrac1{2q'}}\notag\\
    &\leq C \left(\E\Big[ \big|f_R(w_1) f_R(w_2)\big|^{-pq} \Big]\right)^{\tfrac1{q}}
  \end{align*}
  since $f_R'(w)$ is a complex Gaussian with variance $(R^2+R^4|w|^2)e^{R^2|w|^2}$. (The constant $C$ depends on $p,R$ and, if $w_1$ and $w_2$ are restricted to lie in a compact $K$, on $K$.) Applying Lemma~\ref{lem: -pth mom 2 normals} we have
  \begin{equation}\label{eq: pth mom log derivs}
    \E\left[ \Abs{\frac{f_R'(w_1)}{f_R(w_1)} \frac{f_R'(w_2)}{f_R(w_2)}}^p \right] \lesssim (1-|\widehat{K}_R(w_1,w_2)|^2)^{\tfrac1{pq}-1}.
  \end{equation}

  Once more we note that $\Abs{ \nabla\Psi_j(z)}\lesssim \Abs{\frac{f_R'(z)}{f_R(z)}} + R^2|z|$ for $j=1,2$. Therefore to show \eqref{eq: pth mom grad Lj} it suffices to see that
  \begin{equation*}
    \int_{\Gamma_1}\int_{\Gamma_2} \E\left[ \Abs{\frac{f_R'(z_1)}{f_R(z_1)} \frac{f_R'(z_2)}{f_R(z_2)}} \right] \,|dz_1||dz_2| \leq C.
  \end{equation*}
  Now for $z_1\in\Gamma_1$ and $z_2\in\Gamma_2$ we have $1-|\widehat{K}_R(z_1,z_2)|^2=1-e^{-R^2|z_1-z_2|^2}\gtrsim R^2|z_1-z_2|^2$, where the implicit constant depends only on $\Gamma_1$ and $\Gamma_2$. We conclude that, taking $p=1$ in \eqref{eq: pth mom log derivs},
  \begin{equation*}
    \int_{\Gamma_2}\int_{\Gamma_1} \E\left[ \Abs{\frac{f_R'(z_1)}{f_R(z_1)} \frac{f_R'(z_2)}{f_R(z_2)}} \right] \,|dz_1||dz_2| \lesssim \int_{\Gamma_2} \left(\int_{\Gamma_1\setminus\{z_2\}} |z_1-z_2|^{2\left(\tfrac1{q}-1\right)} \,|dz_1|\right) |dz_2|<+\infty,
  \end{equation*}
  since $2(\tfrac1{q}-1)>-1$. This proves \eqref{eq: pth mom grad Lj}.

  Now fix $z_1\in\Gamma_1$ and $z_2\in\Gamma_2$ with $z_1\neq z_2$ and $1< p<2$. Since $|\widehat{K}_R(z_1,z_2)|=e^{-R^2|z_1-z_2|^2/2}<1$ we can find $\eps_0>0$ such that
  \begin{equation*}
    \sup_{\substack{w_1\in D(z_1,\eps_0)\\w_2\in D(z_2,\eps_0)}}|\widehat{K}_R(w_1,w_2)|<1.
  \end{equation*}
  Then, by \eqref{eq: pth mom log derivs},
  \begin{equation*}
    \sup_{\substack{w_1\in D(z_1,\eps_0)\\w_2\in D(z_2,\eps_0)}} \E\left[ \Abs{\frac{f_R'(w_1)}{f_R(w_1)} \frac{f_R'(w_2)}{f_R(w_2)}}^p \right] \lesssim \sup_{\substack{w_1\in D(z_1,\eps_0)\\w_2\in D(z_2,\eps_0)}} (1-|\widehat{K}_R(w_1,w_2)|^2)^{\tfrac1{pq}-1}<+\infty.
  \end{equation*}
  This implies that \eqref{eq: bdd grad p mom} holds, and completes the proof of the lemma.
\end{proof}

\section{The mean and variance}\label{sec: mean and var}
In this section we prove the first part of our theorem, the asymptotic \eqref{eq:var}. We begin by computing the mean of $\Delta_R(\Gamma)$, that is, proving \eqref{eq:mean}; note that by linearity that it's enough to show that
\begin{equation*}
  \E[\Delta_R(\Gamma)]=R^2\im\left(\int_{\Gamma}\bar{z}\,dz\right)
\end{equation*}
for a $\mathcal C^1$ regular oriented simple curve $\Gamma$. For such a curve we have (note that almost surely $f_R$ does not vanish on $\Gamma$)
\begin{equation*}
  \Delta_R(\Gamma)=\im\left(\int_{\Gamma}\frac{f_R'(z)}{f_R(z)}\,dz\right)
\end{equation*}
which implies that
\begin{equation*}
  \E[\Delta_R(\Gamma)]=\im\left(\int_{\Gamma}\E\left[\frac{f_R'(z)}{f_R(z)}\right]\,dz\right);
\end{equation*}
we may apply Fubini by Lemma~\ref{lem: ratio normals}. Applying Lemma~\ref{lem: Kahane mean} we see that
\begin{equation*}
  \E[\Delta_R(\Gamma)]=\im\left(\int_{\Gamma}\frac{R^2\bar{z}e^{R^2|z|^2}}{e^{R^2|z|^2}}\,dz\right)= R^2 \im\left( \int_{\Gamma}\bar{z}\,dz\right),
\end{equation*}
which is precisely \eqref{eq:mean}.

\subsection{The variance}
Given a chain $\Gamma = \sum_i a_i \Gamma_i$, to prove \eqref{eq:var} it is enough to show that
\begin{equation}\label{eq: cov asymptotic}
  \cov (\Delta_R(\Gamma_i),\Delta_R(\Gamma_j))=\frac{\sqrt\pi} 2\zeta\left(\frac32\right)R \cL(\Gamma_i,\Gamma_j)(1+o(1))
\end{equation}
and the rest of this section will be devoted to establishing this asymptotic. First note that we have
\begin{equation}\label{eq: def Delta i}
  \Delta_R(\Gamma_i)=\int_{\Gamma_i}\frac{\partial}{\partial \hat{n}_i}\log|f_R(z)|\,|dz|
\end{equation}
and that \eqref{eq:mean} may be re-written as
\begin{equation}\label{eq: mean Delta i}
  \E[\Delta_R(\Gamma_i)]=\frac12\int_{\Gamma_i}\frac{\partial}{\partial \hat{n}_i}\log K_R(z,z)\,|dz|.
\end{equation}
Recalling that
\begin{equation*}
  \widehat{f}_R(z)=\frac{f_R(z)}{\sqrt{K_R(z,z)}},
\end{equation*}
we see that
\begin{equation*}
  \cov (\Delta_R(\Gamma_i),\Delta_R(\Gamma_j))=\E\left[\int_{\Gamma_i}\int_{\Gamma_j}\frac{\partial^2}{\partial \hat{n}_i\partial \hat{n}_j}\log|\widehat{f}_R(z_j)|\log|\widehat{f}_R(z_i)|\,|dz_j||dz_i|\right].
\end{equation*}
Note that here, and henceforth unless specified otherwise, $\hat{n}_i$ (respectively $\hat{n}_j$) refers to the unit normal vector to the curve $\Gamma_i$ (respectively $\Gamma_j$) at the point $z_i\in\Gamma_i$ (respectively $z_j\in\Gamma_j$).

Now Lemma~\ref{lem: hypoth sat 2 logs} allows us to apply Lemma~\ref{lem: abstract interchange} to see that
\begin{equation}\label{eq: cov ij}
  \cov (\Delta_R(\Gamma_i),\Delta_R(\Gamma_j))=\int_{\Gamma_i}\int_{\Gamma_j}\frac{\partial^2}{\partial \hat{n}_i\partial \hat{n}_j}\E\left[\log|\widehat{f}_R(z_j)|\log|\widehat{f}_R(z_i)|\right]\,|dz_j||dz_i|.
\end{equation}
We add the caveat here (c.f. the remark to Lemma~\ref{lem: abstract interchange}) that the integrand on the right-hand side is defined only for $z_j\neq z_i$. We compute the inner expectation through the following lemma.
\begin{lemma}[\citelist{\cite{SZ}*{Lemma 3.3} \cite{NS}*{Lemma 2.2} \cite{HKPV}*{Lemma 3.5.2}}]\label{lem: dilog}
If $\zeta_1$ and $\zeta_2$ are $\mathcal{N}_\C(0,1)$ random variables with $\E[\zeta_1\bar{\zeta_2}]=\theta$ then
\begin{equation*}
  \E[\log|\zeta_1|]=-\frac\gamma2
\end{equation*}
and
\[
\cov[\log|\zeta_1|,\log|\zeta_2|]= \frac14 \dl(|\theta|^2)
\]
where the dilogarithm is defined by
\[
\dl(z)=\sum_{\alpha=1}^\infty\frac{z^\alpha}{\alpha^2}.
\]
\end{lemma}
Applying the lemma and recalling that
\begin{equation*}
|\widehat{K}_R(z_j,z_i)|=\frac{|K_R(z_j,z_i)|} {\sqrt{K_R(z_j,z_j)K_R(z_i,z_i)}} = e^{-R^2|z_j-z_i|^2/2}
\end{equation*}
we have
\begin{align*}
  \cov (\Delta_R(\Gamma_i),\Delta_R(\Gamma_j))&=\int_{\Gamma_i}\int_{\Gamma_j}\frac{\partial^2}{\partial \hat{n}_i\partial \hat{n}_j} \left(\cov\left(\log|\widehat{f}_R(z_j)|\log|\widehat{f}_R(z_i)|\right) + \frac{\gamma^2}{4}\right) \,|dz_j||dz_i|\\
  &=\frac14\int_{\Gamma_i}\int_{\Gamma_j}\frac{\partial^2}{\partial \hat{n}_i\partial \hat{n}_j}\dl(|\widehat{K}_R(z_j,z_i)|^2)\,|dz_j||dz_i|\\
  &=\frac14\int_{\Gamma_i}\int_{\Gamma_j}\frac{\partial^2}{\partial \hat{n}_i\partial \hat{n}_j}\sum_{\alpha=1}^\infty\frac{e^{-R^2\alpha|z_j-z_i|^2}}{\alpha^2}\,|dz_j||dz_i|.
\end{align*}
(Note that since $\dl$ is not differentiable at $1$, the integrand is still only defined for $z_j\neq z_i$.)

\begin{lemma}\label{cl:abs int}
  \begin{equation*}
    \frac1{R}\int_{\Gamma_i}\int_{\Gamma_j}\Abs{\frac{\partial^2}{\partial \hat{n}_i\partial \hat{n}_j}e^{-R^2|z_j-z_i|^2}}\,|dz_j||dz_i|=O(1)
  \end{equation*}
  where the implicit constant depends only on $\Gamma_i$ and $\Gamma_j$.
\end{lemma}

\begin{lemma}\label{cl:asym of int}
  \begin{equation*}
    \frac1{R}\int_{\Gamma_i}\int_{\Gamma_j}\frac{\partial^2}{\partial \hat{n}_i\partial \hat{n}_j}e^{-R^2|z_j-z_i|^2}\,|dz_j||dz_i|\to2\sqrt{\pi}\cL(\Gamma_i,\Gamma_j)
  \end{equation*}
  as $R\to\infty$.
\end{lemma}

We postpone the proofs of these lemmas, and proceed. Since the power series defining $\dl$ is absolutely convergent on the unit disc we may differentiate termwise to obtain
\begin{equation*}
  \frac{\partial^2}{\partial \hat{n}_i\partial \hat{n}_j}\sum_{\alpha=1}^\infty\frac{e^{-R^2\alpha|z_j-z_i|^2}}{\alpha^2}=\sum_{\alpha=1}^\infty\frac{1}{\alpha^2}\frac{\partial^2}{\partial \hat{n}_i\partial \hat{n}_j}e^{-R^2\alpha|z_j-z_i|^2}
\end{equation*}
for all $z_j\neq z_i$. This implies that, using Lemma~\ref{cl:abs int} and dominated convergence,
\begin{align*}
  \lim_{R\to\infty} \frac1R \cov (\Delta_R(\Gamma_i),\Delta_R(\Gamma_j))& = \frac14\sum_{\alpha=1}^\infty \left( \frac{1}{\alpha^2} \lim_{R\to\infty} \frac1R \int_{\Gamma_i}\int_{\Gamma_j}\frac{\partial^2}{\partial \hat{n}_i\partial \hat{n}_j} e^{-R^2\alpha|z_j-z_i|^2}\,|dz_j||dz_i| \right)\\
  &\overset{\mathrm{Lemma \ref{cl:asym of int}}}{=} \frac14 \sum_{\alpha=1}^\infty \left( \frac{1}{\alpha^2} \sqrt{\alpha}\cdot 2\sqrt{\pi}\cL(\Gamma_i,\Gamma_j)\right)\\
  &=\frac{\sqrt{\pi}}2\left(\sum_{\alpha=1}^\infty\frac{1}{\alpha^{3/2}}\right)\cL(\Gamma_i,\Gamma_j)
\end{align*}
which is \eqref{eq: cov asymptotic}. It remains to prove Lemmas~\ref{cl:abs int} and \ref{cl:asym of int}.

\subsubsection{Proof of Lemmas~\ref{cl:abs int} and \ref{cl:asym of int}}\label{sec: proof of claims}
First note that
\begin{equation*}
  \frac{\partial^2}{\partial \hat{n}_i\partial \hat{n}_j}e^{-R^2|z_j-z_i|^2}=e^{-R^2|z_j-z_i|^2}(2R^2\langle\hat{n}_i,\hat{n}_j\rangle-4R^4\langle\hat{n}_i,z_j-z_i\rangle\langle\hat{n}_j,z_j-z_i\rangle).
\end{equation*}
We will show that
\begin{equation}\label{eq:int precise asym bdd}
  R \int_{\Gamma_i}\int_{\Gamma_j} e^{-R^2|z_j-z_i|^2}\,|dz_j||dz_i|=O(1),
\end{equation}
\begin{equation}\label{eq:int precise asym}
  R \int_{\Gamma_i}\int_{\Gamma_j} \langle\hat{n}_i,\hat{n}_j\rangle e^{-R^2|z_j-z_i|^2}\,|dz_j||dz_i|\to\sqrt{\pi}\cL(\Gamma_i,\Gamma_j),
\end{equation}
\begin{equation}\label{eq:int to 0 bdd}
  R^{3} \int_{\Gamma_i}\int_{\Gamma_j} |z_j-z_i|^2 e^{-R^2|z_j-z_i|^2}\,|dz_j||dz_i|=O(1)
\end{equation}
and that
\begin{equation}\label{eq:int to 0}
  R^{3} \int_{\Gamma_i}\int_{\Gamma_j} \langle\hat{n}_i,z_j-z_i\rangle \langle\hat{n}_j,z_j-z_i\rangle e^{-R^2|z_j-z_i|^2}\,|dz_j||dz_i|\to0.
\end{equation}
This will yield both lemmas, and therefore \eqref{eq:var}.

With this in mind, we define, for $z_i\in\Gamma_i$,
\begin{equation*}
  I_R(z_i)= R \int_{\Gamma_j} \langle\hat{n}_i,\hat{n}_j\rangle e^{-R^2|z_j-z_i|^2}\,|dz_j|,
\end{equation*}
\begin{equation*}
  I_R'(z_i)= R \int_{\Gamma_j} e^{-R^2|z_j-z_i|^2}\,|dz_j|,
\end{equation*}
\begin{equation*}
  J_R(z_i)= R^{3} \int_{\Gamma_j} \langle\hat{n}_i,z_j-z_i\rangle\langle\hat{n}_j,z_j-z_i\rangle e^{-R^2|z_j-z_i|^2}\,|dz_j|
\end{equation*}
and
\begin{equation*}
  J_R'(z_i)= R^{3} \int_{\Gamma_j} |z_j-z_i|^2 e^{-R^2|z_j-z_i|^2}\,|dz_j|.
\end{equation*}
Fix $\beta>3$, define
\begin{equation*}
\eps_R=\frac{\sqrt{\beta\log R}}{R}
\end{equation*}
and split
\begin{equation*}\Gamma_i=\Gamma_i'\cup\Gamma_i''\cup(\Gamma_i\cap\Gamma_j)
\end{equation*}
where
\begin{equation*}\Gamma_i'=\{z_i\in\Gamma_i\colon d(z_i,\Gamma_j)\geq \eps_R\}
\end{equation*}
and
\begin{equation*}\Gamma_i''=\{z_i\in\Gamma_i\colon 0<d(z_i,\Gamma_j)< \eps_R\}
\end{equation*}
(see Figure~\ref{fig: Gamma i decom}); these sets (and all of the sets we define subsequently) may be empty.
\begin{figure}
  \centering
  \includegraphics[width=0.5\columnwidth]{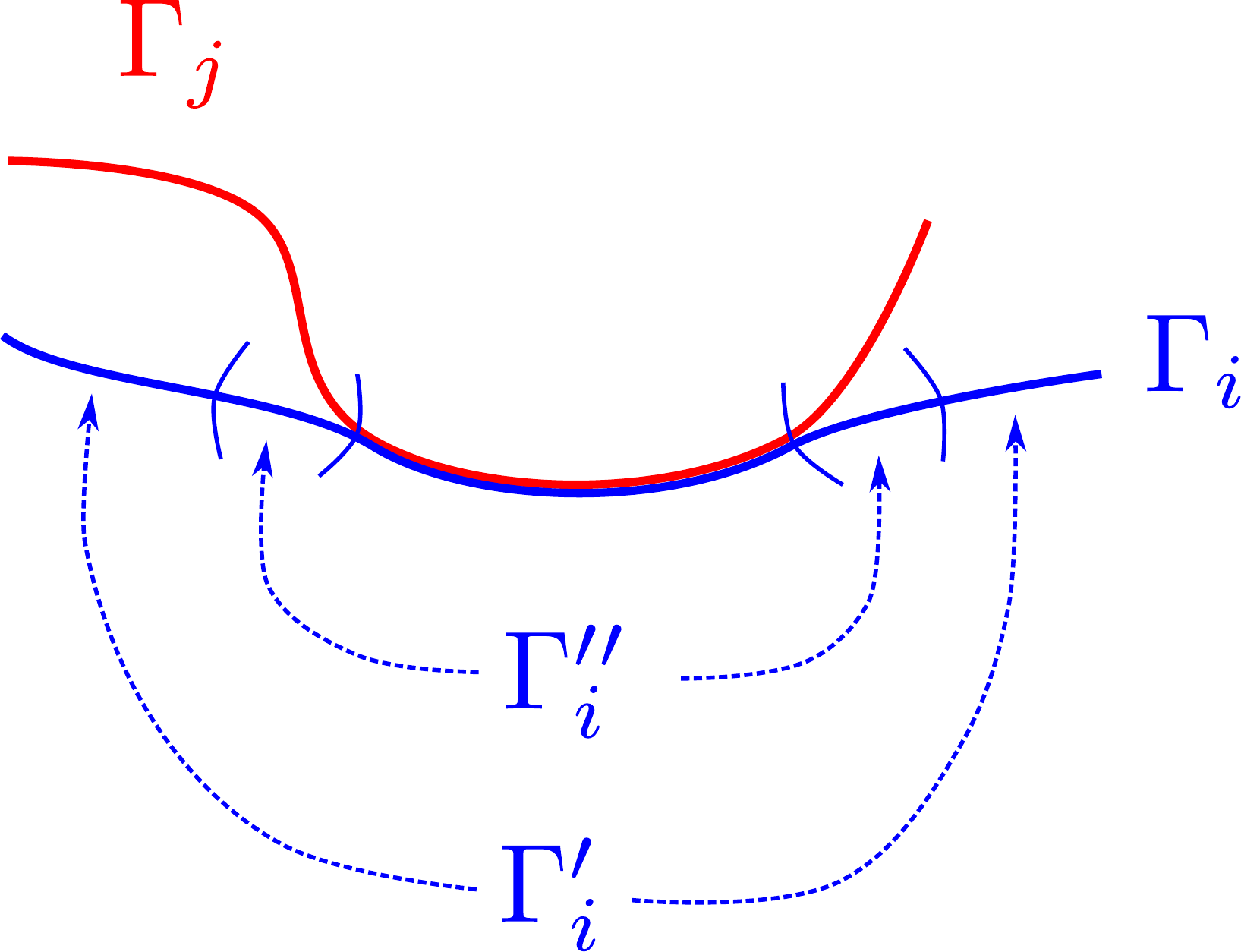}
  \caption{Illustration of $\Gamma_i'$ and $\Gamma_i''$}
  \label{fig: Gamma i decom}
\end{figure}

\paragraph{{\bf Estimating $J_R'$}}\hspace{0pt} \\
We begin by estimating $J_R'$. We estimate separately the integral of $J_R'$ over each of the sets $\Gamma_i',\Gamma_i''$ and $\Gamma_i\cap\Gamma_j$. Trivially
\begin{equation*}
  \int_{\Gamma_i'}J_R'(z_i)|dz_i|\leq CR^{3 - \beta}
\end{equation*}
where the constant $C$ depends only on $\Gamma_i$ and $\Gamma_j$. Next, for $z_i\in\Gamma_i''$, denote by $z_i^*$ the closest point on $\Gamma_j$ to $z_i$ (if there is more than one such point, we choose one arbitrarily). Fix $z_i\in\Gamma_i''$ and define the points in $\Gamma_j$ that are ``far'' from $z_i^*$ by
\begin{equation*}
  \Gamma_j^{(F)}(z_i)=\left\{z_j\in\Gamma_j\colon |z_j-z_i^*|>\frac2R\right\}
\end{equation*}
and the ``nearby'' points by
\begin{equation*}
  \Gamma_j^{(N)}(z_i)=\left\{z_j\in\Gamma_j\colon |z_j-z_i^*|\leq\frac2{R}\right\};
\end{equation*}
see Figure~\ref{fig: Gamma j N}. We split
\begin{equation*}
  J_R'(z_i)= R \left(\int_{\Gamma_j^{(F)}}+\int_{\Gamma_j^{(N)}}\right) R^2|z_j-z_i|^2 e^{-R^2|z_j-z_i|^2}\,|dz_j|
\end{equation*}
and estimate each integral separately. Note that
\begin{equation*}
  R \int_{\Gamma_j^{(N)}} R^2|z_j-z_i|^2 e^{-R^2|z_j-z_i|^2}\,|dz_j|\lesssim R \;\mathrm{length}(\Gamma_j^{(N)})=O(1).
\end{equation*}
Note also that $|z_j-z_i|\geq|z_j-z_i^*|-|z_i^*-z_i|\geq|z_j-z_i^*|-|z_j-z_i|$, for any $z_j\in\Gamma_j$, by the definition of $z_i^*$. This implies that $R |z_j-z_i|\geq\frac {R} 2|z_j-z_i^*|>1$ for $z_i\in\Gamma_i''$ and $z_j\in\Gamma_j^{(F)}$ and so, since the function $t\mapsto t^2e^{-t^2}$ is decreasing for $t>1$, we have
\begin{equation*}
  R \int_{\Gamma_j^{(F)}} R^2|z_j-z_i|^2 e^{-R^2|z_j-z_i|^2}\,|dz_j|\leq\frac{R}4\int_{\Gamma_j^{(F)}} R^2|z_j-z_i^*|^2 e^{-\tfrac{R^2}4|z_j-z_i^*|^2}\,|dz_j|.
\end{equation*}

\begin{figure}
  \centering
  \includegraphics[width=0.4\columnwidth]{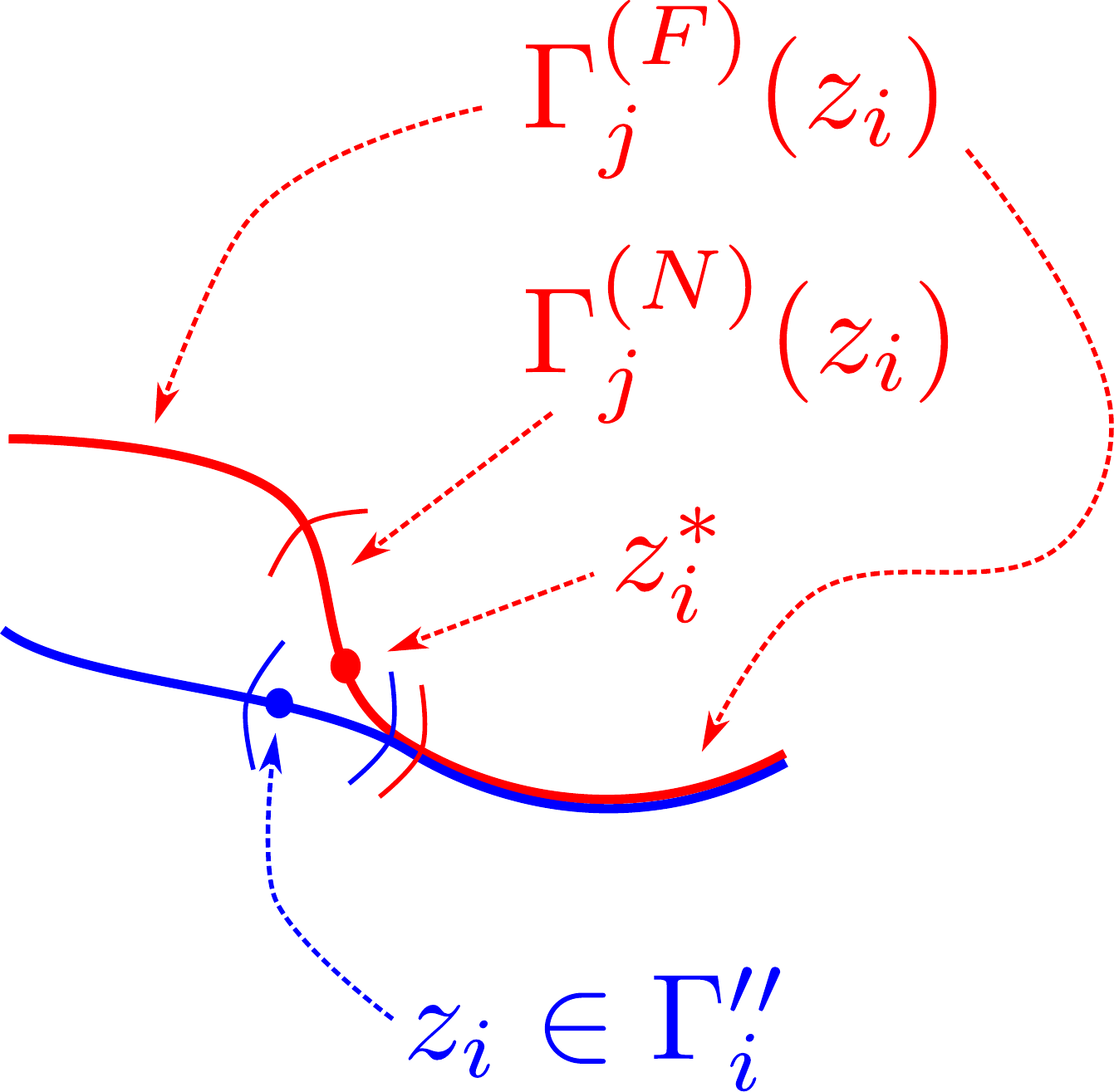}
  \caption{Illustration of $\Gamma_j^{(F)}(z_i)$ and $\Gamma_j^{(N)}(z_i)$}
  \label{fig: Gamma j N}
\end{figure}

We now use some ``Laplace type estimates'' to bound the integral on the right-hand side of this previous inequality. Let $\gamma_j\colon[0,1]\to\C$ be a parameterisation of the curve $\Gamma_j$ satisfying $0<m\leq|\dot{\gamma}(t)|\leq M$ for all $t\in[0,1]$. Since $\Gamma_j$ is simple, we see that there exists $m'>0$ such that $m'|t-s|\leq|\gamma_j(t)-\gamma_j(s)|\leq M|t-s|$. Denote by $t^*$ the (unique) value such that $\gamma_j(t^*)=z_i^*$ and note that\footnote{Strictly speaking we should write $\left\{z_j=\gamma_j(t)\colon 0\leq t\leq1\text{ and }|t-t^*|>\frac2{MR}\right\}$; we shall frequently ignore this issue, as it will not affect our upper bounds.}
\begin{equation*}
  \Gamma_j^{(F)}(z_i)\subseteq\left\{z_j=\gamma_j(t)\colon |t-t^*|>\frac2{MR}\right\}.
\end{equation*}
We thus have
\begin{align*}
  R\int_{\Gamma_j^{(F)}} R^2|z_j-z_i^*|^2 &e^{-\tfrac{R^2}4|z_j-z_i^*|^2}\,|dz_j|\\
  &\leq R\int_{|t-t^*|>\frac2{MR}} R^2|\gamma_j(t)-\gamma_j(t^*)|^2 e^{-\tfrac {R^2}4|\gamma_j(t)-\gamma_j(t^*)|^2}\Abs{\dot{\gamma_j}(t)}\,dt\\
  &\leq R\int_{|t-t^*|>\frac2{MR}} R^2M^2(t-t^*)^2 e^{-\tfrac {R^2}4(m')^2(t-t^*)^2}M\,dt,
\end{align*}
and making the change of variables $s=\tfrac{R}2m'(t-t^*)$ we have
\begin{equation*}
  R\int_{\Gamma_j^{(F)}} R^2|z_j-z_i^*|^2 e^{-4R^2|z_j-z_i^*|^2}\,|dz_j|\leq\frac{8M^3}{(m')^3}\int_{|s|>\tfrac{m'}M} s^2 e^{-s^2}\,ds=O(1).
\end{equation*}

This implies that
\begin{equation*}
  \int_{\Gamma_i''}J_R'(z_i)\,|dz_i|\lesssim\mathrm{length}(\Gamma_i'')\to0
\end{equation*}
as $R\to\infty$, since the set $\Gamma_i''$ decreases to the empty set as $\eps_R\to0$. We thus have
\begin{equation}\label{eq: J' outside int}
  \int_{\Gamma_i'\cup\Gamma_i''}J_R'(z_i)\,|dz_i|\to0.
\end{equation}

We now bound $\int_{\Gamma_i\cap\Gamma_j}J_R'(z_i)\,|dz_i|$. Fixing $z_i$, it is clear that we may ignore the points $z_j\in\Gamma_j$ where $|z_j-z_i|\geq \eps_R$, since their contribution is uniformly negligible. Denote the points ``close'' to $z_i$ by $\Gamma_j^{(C)}(z_i)=\{z_j\in\Gamma_j\colon |z_j-z_i|\leq\eps_R\}$, see Figure~\ref{fig: Gamma j C}.

\begin{figure}
  \centering
  \includegraphics[width=0.4\columnwidth]{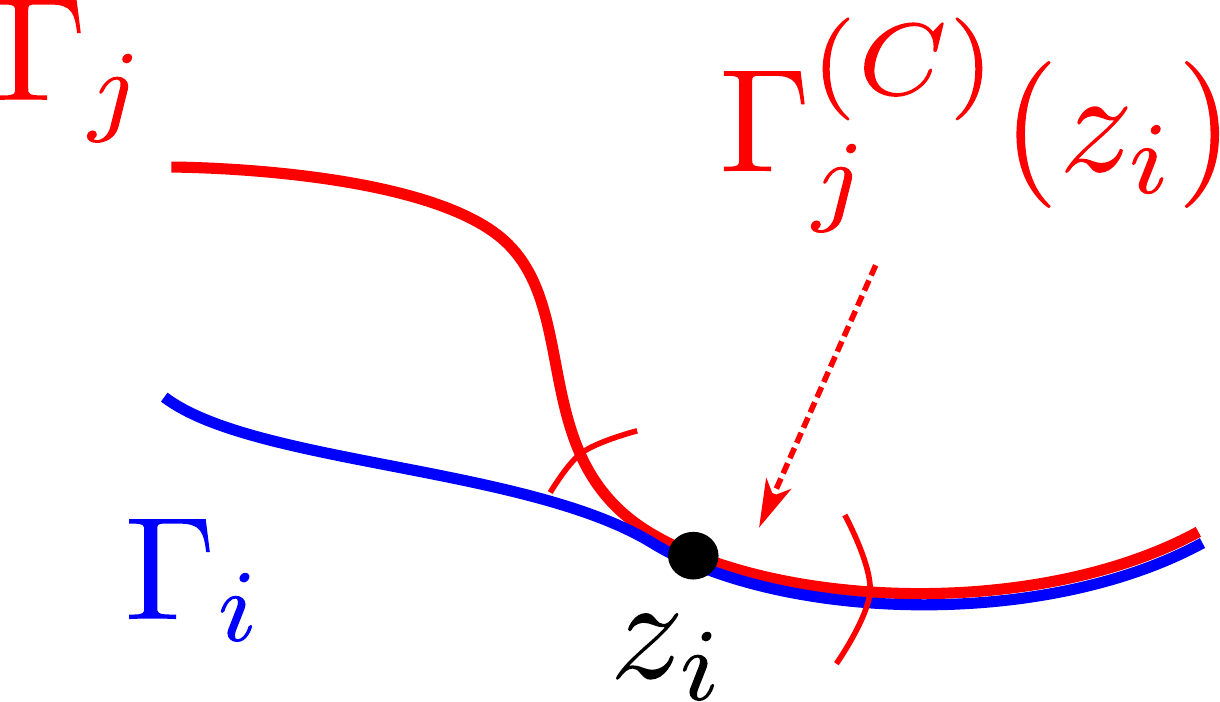}
  \caption{Illustration of $\Gamma_j^{(C)}(z_i)$}
  \label{fig: Gamma j C}
\end{figure}

Let $\gamma_j$ be the same parameterisation of $\Gamma_j$ as before and let $\tau(z_i)$ be the (unique) value such that $\gamma_j(\tau(z_i))=z_i$. Arguing similarly we get
\begin{align}\label{eq: J' on int }
  R\int_{\Gamma_j^{(C)}} R^2|z_j-z_i|^2 &e^{-R^2|z_j-z_i|^2}\,|dz_j|\notag\\
  &\leq R\int_{|t-\tau(z_i)|\leq\frac{\eps_R}{m'}} R^2|\gamma_j(t)-\gamma_j(\tau(z_i))|^2 e^{- R^2|\gamma_j(t)-\gamma_j(\tau(z_i))|^2}\Abs{\dot{\gamma_j}(t)}\,dt\\
  &\leq\left(\frac M{m'}\right)^3\int_{|s|\leq R\eps_R} s^2 e^{-s^2}\,ds=O(1)\notag,
\end{align}
which shows that $J_R'(z_i)$ is bounded for $z_i\in\Gamma_i\cap\Gamma_j$ and so, when combined with \eqref{eq: J' outside int} proves \eqref{eq:int to 0 bdd}.

\paragraph{{\bf Estimating $J_R$}}\hspace{0pt} \\
We next show \eqref{eq:int to 0}. Since trivially $|J_R(z_i)|\leq J_R'(z_i)$ we see that \eqref{eq: J' outside int} implies that
\begin{equation*}
  \Abs{\int_{\Gamma_i'\cup\Gamma_i''}J_R(z_i)\,|dz_i|}\leq\int_{\Gamma_i'\cup\Gamma_i''}J_R'(z_i)\,|dz_i|\to0\quad\text{as }\eps_R\to0.
\end{equation*}
Furthermore, for a fixed $z_i\in\Gamma_i\cap\Gamma_j$, as we noted previously
\begin{equation*}
  R\int_{\Gamma_j\setminus\Gamma_j^{(C)}} R^2|z_j-z_i|^2 e^{-R^2|z_j-z_i|^2}\,|dz_j|
\end{equation*}
is uniformly negligible. Finally note that for $z_i\in\Gamma_i\cap\Gamma_j$ and $z_j\in\Gamma_j^{(C)}(z_i)$,
\begin{align*}
  \left\langle\hat{n}_j(z_j),z_j-z_i\right\rangle&=|z_j-z_i|\left\langle\hat{n}_j(z_j),\frac{z_j-z_i}{|z_j-z_i|}\right\rangle\\
  &=|z_j-z_i|\left\langle\hat{n}_j(z_i),\hat{\tau}_j(z_i)\right\rangle(1+o(1))=o(|z_j-z_i|)\quad\text{as }\eps_R\to0
\end{align*}
where $\hat{\tau}_j$ denotes the unit tangent vector to $\Gamma_j$, and the estimate is uniform in $z_i$. This implies that for $z_i\in\Gamma_i\cap\Gamma_j$ we have $|J_R(z_i)|=o(J_R'(z_i))=o(1)$, by \eqref{eq: J' on int }.
This proves \eqref{eq:int to 0}.

\paragraph{{\bf Estimating $I_R'$}}\hspace{0pt} \\
We next show \eqref{eq:int precise asym bdd}, the argument is similar to the proof of \eqref{eq:int to 0 bdd}. It is again easy to see that
\begin{equation*}
  \int_{\Gamma_i'}I_R'(z_i)\,|dz_i|\leq CR^{1-\beta}\to0
\end{equation*}
while, for $z_i\in\Gamma_i''$, using the same notation as before, since $|z_j-z_i|\geq\tfrac12|z_j-z_i^*|$ we get
\begin{align*}
  I_R'(z_i)&=R\int_{\Gamma_j} e^{-R^2|z_j-z_i|^2}\,|dz_j| \leq R\int_{\Gamma_j} e^{-\tfrac {R^2}4|z_j-z_i^*|^2}\,|dz_j|\\
  &=R\int_0^1 e^{-\tfrac {R^2}4|\gamma_j(t)-\gamma_j(t^*)|^2}\Abs{\dot{\gamma_j}(t)}\,dt \leq \frac{2M}{m'}\int_{-\infty}^\infty  e^{-s^2}\,ds = O(1).
\end{align*}
We therefore have
\begin{equation*}
  \int_{\Gamma_i''}I_R'(z_i)\,|dz_i|\leq C\,\mathrm{length}(\Gamma_i'')\to0\quad\text{as }\eps_R\to0.
\end{equation*}
Finally, for $z_i\in\Gamma_i\cap\Gamma_j$, using again the same notation, it is easy to see once more that the contribution to $I_R'(z_i)$ of $\Gamma_j\setminus\Gamma_j^{(C)}(z_i)$ is negligible and that
\begin{align*}
  R\int_{\Gamma_j^{(C)}} e^{-R^2|z_j-z_i|^2}\,|dz_j|
  &\leq R\int_{|t-\tau(z_i)|\leq\frac{\eps_R}{m'}} e^{- R^2|\gamma_j(t)-\gamma_j(\tau(z_i))|^2}\Abs{\dot{\gamma_j}(t)}\,dt\\
  &\leq\frac M{m'}\int_{|s|\leq R\eps_R} e^{-s^2}\,ds=O(1).
\end{align*}
We have shown that $I_R'(z_i)=O(1)$ for $z_i\in\Gamma_i\cap\Gamma_j$, which proves \eqref{eq:int precise asym bdd}.

\paragraph{{\bf Asymptotic for $I_R$}}\hspace{0pt} \\
It remains to prove \eqref{eq:int precise asym}. Note that
\begin{equation*}
  \Abs{\int_{\Gamma_i'\cup\Gamma_i''}I_R(z_i)\,|dz_i|}\leq\int_{\Gamma_i'\cup\Gamma_i''}I_R'(z_i)\,|dz_i|\to0\quad\text{as }\eps_R\to0
\end{equation*}
and so it remains only to compute
\begin{equation*}
  \int_{\Gamma_i\cap\Gamma_j}I_R(z_i)\,|dz_i|.
\end{equation*}
If the curve $\Gamma_j$ is not closed we define $z_j^+$ and $z_j^-$ to be the endpoints of the curve and $\Gamma_j^\pm=\{z_j\in\Gamma_j\colon |z_j-z_j^\pm|< \eps_R\}$; if the curve is closed we define these sets to be empty. Note once more that
\begin{equation*}
  \Abs{\int_{\Gamma_i\cap\Gamma_j\cap\Gamma_j^\pm}I_R(z_i)\,|dz_i|}\leq \int_{\Gamma_i\cap\Gamma_j\cap\Gamma_j^\pm}I_R'(z_i)\,|dz_i|\lesssim\mathrm{length}(\Gamma_j^\pm)\to0\quad\text{as }\eps_R\to0.
\end{equation*}
For $z_i\in(\Gamma_i\cap\Gamma_j)\setminus(\Gamma_j^+\cup\Gamma_j^-)$ recall that $\Gamma_j^{(C)}(z_i)=\{z_j\in\Gamma_j\colon |z_j-z_i|\leq\eps_R\}$ and note that for $z_j\in\Gamma_j^{(C)}(z_i)$ we have $|\hat{n}_j(z_j)-\hat{n}_j(z_i)|=o(1)$, where the term $o(1)$ is uniform in $z_i$. We therefore have
\begin{align}\label{eq: precise without negl}
  I_R(z_i)&=R\int_{\Gamma_j^{(C)}(z_i)} \langle\hat{n}_i(z_i),\hat{n}_j(z_j)\rangle e^{-R^2|z_j-z_i|^2}\,|dz_j|\notag\\
  &\qquad\qquad\qquad\qquad+R\int_{\Gamma_j\setminus\Gamma_j^{(C)}(z_i)} \langle\hat{n}_i(z_i),\hat{n}_j(z_i)\rangle e^{-R^2|z_j-z_i|^2}\,|dz_j|\notag\\
  &=R\langle\hat{n}_i(z_i),\hat{n}_j(z_i)\rangle\int_{\Gamma_j^{(C)}(z_i)} e^{-R^2|z_j-z_i|^2}\,|dz_j|(1+o(1))+O(R^{1-\beta}),
\end{align}
and we shall compute the asymptotics of this last integral using (more accurate) ``Laplace type estimates''.

Let $\gamma_j\colon[0,1]\to\C$ be the same parameterisation of $\Gamma_j$ as before, and let $\tau(z_i)$ be the value such that $\gamma_j(\tau(z_i))=z_i$. Note that
\begin{equation*}
  \left\{z_j=\gamma_j(t)\colon |t-\tau(z_i)|\leq\frac{\eps_R}{M}\right\}\subseteq\Gamma_j^{(C)}(z_i)\subseteq\left\{z_j=\gamma_j(t)\colon |t-\tau(z_i)|\leq\frac{\eps_R}{m'}\text{ and }0<t<1\right\}.
\end{equation*}
We remark here that, if $\{t_+,t_-\}=\{0,1\}$, then $\{\gamma_j(t_+),\gamma_j(t_-)\}=\{z_j^+,z_j^-\}$. Since
\begin{equation*}
\eps_R\leq|\gamma_j(\tau(z_i))-\gamma_j(t_\pm)|\leq M|\tau(z_i) - t_\pm|
\end{equation*}
we have $\tau(z_i)\geq\eps_R/M$ and $\tau(z_i)\leq1-\eps_R/M$ and so the range $|t-\tau(z_i)|\leq\tfrac{\eps_R}{M}$ is contained in $[0,1]$. The above implies that, given $0<\delta<1$, for large enough $R$ (uniformly in $\tau(z_i)$) we have
\begin{equation*}
  (1-\delta)|\dot{\gamma}_j(\tau(z_i))||t-\tau(z_i)|\leq|\gamma_j(t)-\gamma_j(\tau(z_i))|\leq(1+\delta)|\dot{\gamma}_j(\tau(z_i))||t-\tau(z_i)|
\end{equation*}
and
\begin{equation*}
  (1-\delta)|\dot{\gamma}_j(\tau(z_i))|\leq|\dot{\gamma}_j(t)|\leq(1+\delta)|\dot{\gamma}_j(\tau(z_i))|
\end{equation*}
for $t$ such that $\gamma_j(t)\in\Gamma_j^{(C)}(z_i)$. This implies that, defining
\begin{equation*}
\tau_+=\min\{1,\tau(z_i)+\tfrac{\eps_R}{m'}\}\quad\text{ and }\quad\tau_-=\max\{0,\tau(z_i)-\tfrac{\eps_R}{m'}\},
\end{equation*}
we have
\begin{align*}
  \int_{\Gamma_j^{(C)}(z_i)} e^{-R^2|z_j-z_i|^2}\,|dz_j|& \leq \int_{\tau_-}^{\tau_+} e^{-R^2|\gamma_j(t)-\gamma_j(\tau(z_i))|^2}|\dot{\gamma}_j(t)| \,dt\\
  &\leq (1+\delta) \int_{\tau_-}^{\tau_+} e^{-R^2 (1-\delta)^2 |\dot{\gamma}_j(\tau(z_i))|^2 (t-\tau(z_i))^2} |\dot{\gamma}_j(\tau(z_i))| \,dt
\end{align*}
and the change of variables $s=R(1-\delta)|\dot{\gamma}_j(\tau(z_i))|(t-\tau(z_i))$ yields
\begin{equation*}
  \int_{\Gamma_j^{(C)}(z_i)} e^{-R^2|z_j-z_i|^2}\,|dz_j| \leq(1+\delta)\int_{R(1-\delta)|\dot{\gamma}_j(\tau(z_i))|(\tau_{-}-\tau(z_i))}^{R(1-\delta)|\dot{\gamma}_j(\tau(z_i))|(\tau_{+}-\tau(z_i))} e^{-s^2} \frac{ds}{R(1-\delta)}.
\end{equation*}
Notice that
\begin{equation*}
\tau_{+}-\tau(z_i)=\min\{1-\tau(z_i),\tfrac{\eps_R}{m'}\}\geq\tfrac{\eps_R}{M}\quad\text{ and }\quad\tau_{-}-\tau(z_i)=\max\{-\tau(z_i),-\tfrac{\eps_R}{m'}\}\leq-\tfrac{\eps_R}{M}.
\end{equation*}
Since $R\eps_R\to\infty$ (and $|\dot{\gamma}_j(\tau(z_i))|\neq0$) the right hand side of the previous displayed expression equals
\begin{equation*}
  \frac{(1+\delta)}{R(1-\delta)}\int_{\R}e^{-s^2}\,ds(1+o(1)) = \frac{(1+\delta)\sqrt{\pi}}{R(1-\delta)}(1+o(1)).
\end{equation*}

Similar computations yield
\begin{equation*}
  \int_{\Gamma_j^{(C)}(z_i)} e^{-R^2|z_j-z_i|^2}\,|dz_j| \geq \frac{(1-\delta)\sqrt{\pi}}{R(1+\delta)}(1+o(1))
\end{equation*}
and since $\delta$ is arbitrary we conclude that
\begin{equation*}
  \int_{\Gamma_j^{(C)}(z_i)} e^{-R^2|z_j-z_i|^2}\,|dz_j| = \frac{\sqrt{\pi}}{R}(1+o(1)).
\end{equation*}
Combining this with \eqref{eq: precise without negl}, and discarding the integration over $\Gamma_j\setminus\Gamma_j^{(C)}(z_i)$, we have
\begin{equation*}
  I_R(z_i)=\sqrt{\pi} \langle\hat{n}_i(z_i),\hat{n}_j(z_i)\rangle(1+o(1)),
\end{equation*}
where the term $o(1)$ is uniform in $z_i$. We conclude that
\begin{align*}
  \int_{\Gamma_i}I_R(z_i)\,|dz_i|&=\sqrt \pi\int_{(\Gamma_i\cap\Gamma_j)\setminus(\Gamma_j^+\cup\Gamma_j^-)}\langle\hat{n}_i(z_i),\hat{n}_j(z_i)\rangle\,|dz_i|(1+o(1))\\
  &=\sqrt \pi\left(\cL(\Gamma_i,\Gamma_j)-\int_{(\Gamma_i\cap\Gamma_j)\cap(\Gamma_j^+\cup\Gamma_j^-)}\langle\hat{n}_i(z_i),\hat{n}_j(z_i)\rangle\,|dz_i|\right)(1+o(1))\\
  &=(\sqrt \pi+o(1))\cL(\Gamma_i,\Gamma_j),
\end{align*}
which is \eqref{eq:int precise asym}. This completes the proof of the lemmas, and therefore of \eqref{eq:var}.

\section{Asymptotic Normality}
In this section we show that $\Delta_R(\Gamma)$ is asymptotically normal, which will complete the proof of Theorem~\ref{thm: main}. We first define a random variable $\Delta_R^{(m)}$ that approximates $\Delta_R(\Gamma)$ in $L^2(\Pro)$ and then prove a CLT for $\Delta_R^{(m)}$. Specifically, defining $\overline{\Delta}_R= \Delta_R(\Gamma) - \E[\Delta_R(\Gamma)]$, we will show that:
\begin{itemize}
  \item There exists $R_0$ such that
        \begin{equation}\label{eq: L2 approx for CLT}
          \E\big[(\overline{\Delta}_R-\Delta_R^{(m)})^2\big]\leq C \frac{R}{\sqrt{m}}
        \end{equation}
        for all $R\geq R_0$ and $m\geq1$.
  \item For each fixed $m$
        \begin{equation}\label{eq: CLT m}
          \frac{\Delta_R^{(m)}}{\sqrt{\var \Delta_R^{(m)}}} \to \cN_\R(0,1)
        \end{equation}
        in distribution, as $R\to\infty$.
\end{itemize}
When combined with our previous asymptotic for the variance, this allows us to conclude asymptotic normality for $\Delta_R(\Gamma)$, by a standard argument. We begin by defining the approximant $\Delta_R^{(m)}$.

\subsection{Definition of approximant}
Recall that
\begin{equation*}
\overline{\Delta}_R = \sum_{i=1}^N a_i\int_{\Gamma_i}\frac{\partial}{\partial \hat{n}_i}\log|\widehat{f}_R(z_i)|\,|dz_i|,
\end{equation*}
this follows from \eqref{eq: def Delta i} and \eqref{eq: mean Delta i}. We will define
\begin{equation*}
  \Delta_R^{(m)} = \sum_{i=1}^N a_i \Delta_i^{(m)} = \sum_{i=1}^N a_i\int_{\Gamma_i}\frac{\partial}{\partial \hat{n}_i}\log_m|\widehat{f}_R(z_i)|\,|dz_i|
\end{equation*}
where $\log_m$ is a polynomial that approximates $\log$ in an appropriate sense. To this end we recall the Wiener chaos decomposition (sometimes called the Hermite-It\={o} expansion) of $L^2(\mu)$ where $d\mu(z)=\tfrac1\pi e^{-|z|^2}dm(z)$ is the Gaussian measure on the plane; for a more comprehensive treatment we refer the reader to \cite{Jan}*{Chapters 2 and 3}.

Let $\mathcal P_m$ denote the subspace of $L^2(\mu)$ given by polynomials (in the variables $z$ and $\bar{z}$) of degree at most $m$, and denote $H^{\Wi{0}}=\mathcal P_0$ and $H^{\Wi{m}}=\mathcal P_m\ominus\mathcal P_{m-1}$ for $m\geq1$. Given a monomial $\zeta^\alpha\bar{\zeta}^\beta$ with $\alpha+\beta=m$ we write $\Wi{\zeta^\alpha\bar{\zeta}^\beta}$ to denote its projection to $H^{\Wi{m}}$, which is usually called a Wick product. A computation (see \cite{Jan}*{Example 3.32}) shows that the set of all Wick products $\Wi{\zeta^\alpha\bar{\zeta}^\beta}$ with $\alpha+\beta=m$ is an orthogonal basis for $H^{\Wi{m}}$, and moreover $\|\Wi{\zeta^\alpha\bar{\zeta}^\beta}\|^2=\alpha!\beta!$ (the norm here is the norm inherited from $L^2(\mu)$). Furthermore \cite{Jan}*{Theorem 2.6}
\begin{equation*}
L^2(\mu)=\bigoplus_{m=0}^\infty H^{\Wi{m}}.
\end{equation*}

We now expand $\log|\zeta|$ in terms of this orthonormal basis. Since the function is radial, only the terms with $\alpha=\beta$ contribute, and a calculation \cite{NS}*{Lemma 2.1} yields
\begin{equation}\label{eq: Wick log}
  \log|\zeta|=-\frac \gamma 2 + \sum_{\alpha=1}^\infty \frac{c_\alpha}{\alpha!}\Wiaa{\zeta}{\alpha}
\end{equation}
where $c_\alpha=\frac{(-1)^{\alpha+1}}{2\alpha}$.
\begin{rem}
  We may alternatively interpret \eqref{eq: Wick log} as an expansion of the logarithm in terms of Laguerre polynomials, by noting that $\Wiaa{\zeta}{\alpha}=(-1)^\alpha L_{\alpha}(|\zeta|^2)$ where $L_\alpha(x)=e^x\frac{d^\alpha}{dx^\alpha}(x^\alpha e^{-x})$.
\end{rem}

We finally define
\begin{equation*}
  \log_m|\zeta|=-\frac \gamma 2 + \sum_{\alpha=1}^m \frac{c_\alpha}{\alpha!}\Wiaa{\zeta}{\alpha}
\end{equation*}
which defines $\Delta_R^{(m)}$. Note that $\log_m|\widehat{f}_R(z)|$ approximates $\log|\widehat{f}_R(z)|$ in $L^2(\Pro)$.

\subsection{Quantifying the approximation}
We first define
\begin{equation*}
\Delta_i = \int_{\Gamma_i}\frac{\partial}{\partial \hat{n}_i}\log|\widehat{f}_R(z_i)|\,|dz_i|
\end{equation*}
and note that
\begin{equation*}
  \E\big[ ( \overline{\Delta}_R-\Delta_R^{(m)} )^2 \big] = \sum_{i,j=1}^N a_i a_j \E\big[ (\Delta_i - \Delta_i^{(m)})(\Delta_j - \Delta_j^{(m)})\big].
\end{equation*}
We have already computed (see \eqref{eq: cov ij}) that
\begin{equation*}
  \E[ \Delta_i \Delta_j ]= \int_{\Gamma_i}\int_{\Gamma_j} \frac{\partial^2}{\partial \hat{n}_i\partial \hat{n}_j} \E\left[\log|\widehat{f}_R(z_j)|\log|\widehat{f}_R(z_i)|\right]\,|dz_j||dz_i|.
\end{equation*}
Also, since
\begin{equation*}
  \E[ \Delta_i \Delta_j^{(m)} ]= \E\left[\int_{\Gamma_i}\int_{\Gamma_j} \frac{\partial^2}{\partial \hat{n}_i\partial \hat{n}_j} \log_m|\widehat{f}_R(z_j)|\log|\widehat{f}_R(z_i)|\,|dz_j||dz_i|\right],
\end{equation*}
Lemma~\ref{lem: hypoth sat log poly} allows us to apply Lemma~\ref{lem: abstract interchange} to see that
\begin{equation*}
  \E[ \Delta_i \Delta_j^{(m)} ]= \int_{\Gamma_i}\int_{\Gamma_j} \frac{\partial^2}{\partial \hat{n}_i\partial \hat{n}_j}\E\left[ \log_m|\widehat{f}_R(z_j)|\log|\widehat{f}_R(z_i)|\right]\,|dz_j||dz_i|.
\end{equation*}
Arguing identically, but using Lemma~\ref{lem: hypoth sat polys} (with $N=2$) instead of Lemma~\ref{lem: hypoth sat log poly}, we get
\begin{equation}\label{eq: cov Delta m}
  \E[ \Delta_i^{(m)} \Delta_j^{(m)} ]= \int_{\Gamma_i}\int_{\Gamma_j} \frac{\partial^2}{\partial \hat{n}_i\partial \hat{n}_j}\E\left[ \log_m|\widehat{f}_R(z_j)|\log_m|\widehat{f}_R(z_i)|\right]\,|dz_j||dz_i|.
\end{equation}
We conclude that
\begin{align*}
  \E\big[ (\Delta_i - \Delta_i^{(m)})&(\Delta_j - \Delta_j^{(m)})\big]\\
  &= \int_{\Gamma_i}\int_{\Gamma_j} \frac{\partial^2}{\partial \hat{n}_i\partial \hat{n}_j}\E\Big[ (\log|\widehat{f}_R(z_j)| - \log_m|\widehat{f}_R(z_j)|)\cdot\\
  &\qquad\qquad\qquad\qquad\qquad\qquad \cdot(\log|\widehat{f}_R(z_i)| - \log_m|\widehat{f}_R(z_i)| ) \Big]\,|dz_j||dz_i|\\
  &= \int_{\Gamma_i}\int_{\Gamma_j} \frac{\partial^2}{\partial \hat{n}_i\partial \hat{n}_j} \E\Big[ \bigg(\sum_{\alpha_j>m} \frac{c_{\alpha_j}}{\alpha_j!}\Wiaa{\widehat{f}_R(z_j)}{\alpha_j}\bigg)\cdot\\
  &\qquad\qquad\qquad\qquad\qquad\qquad \cdot \bigg(\sum_{\alpha_i>m} \frac{c_{\alpha_i}}{\alpha_i!}\Wiaa{\widehat{f}_R(z_i)}{\alpha_i}\bigg) \Big]\,|dz_j||dz_i|.
\end{align*}

Now since the expansions inside the expectation are valid in $L^2(\Pro)$ we have
\begin{align*}
  \E\Big[ \bigg(\sum_{\alpha_j>m} \frac{c_{\alpha_j}}{\alpha_j!}\Wiaa{\widehat{f}_R(z_j)}{\alpha_j}\bigg) &\bigg(\sum_{\alpha_i>m} \frac{c_{\alpha_i}}{\alpha_i!}\Wiaa{\widehat{f}_R(z_i)}{\alpha_i}\bigg) \Big]\\
  &=\sum_{\alpha_i,\alpha_j>m} \frac{c_{\alpha_i}c_{\alpha_j}}{\alpha_i!\alpha_j!}\E\left[\Wiaa{\widehat{f}_R(z_j)}{\alpha_j}\, \Wiaa{\widehat{f}_R(z_i)}{\alpha_i} \right].
\end{align*}
Now, by \cite{Jan}*{Theorem 3.9}, we have
\begin{equation}\label{eq: exp wick prod}
  \E\left[\Wiaa{\widehat{f}_R(z_j)}{\alpha_j}\, \Wiaa{\widehat{f}_R(z_i)}{\alpha_i} \right]=
  \begin{cases}
    \alpha_i!\alpha_j!|\widehat{K}_R(z_j,z_i)|^{2\alpha_i}&\text{ if }\alpha_i=\alpha_j\\
    0&\text{otherwise}
  \end{cases}
\end{equation}
which yields
\begin{align*}
  \E\big[ (\Delta_i - \Delta_i^{(m)}) (\Delta_j - \Delta_j^{(m)})\big] &= \int_{\Gamma_i}\int_{\Gamma_j} \frac{\partial^2}{\partial \hat{n}_i\partial \hat{n}_j} \sum_{\alpha>m} c_{\alpha}^2 |\widehat{K}_R(z_j,z_i)|^{2\alpha}\,|dz_j||dz_i|\\
  &=\frac14 \int_{\Gamma_i}\int_{\Gamma_j} \frac{\partial^2}{\partial \hat{n}_i\partial \hat{n}_j} \sum_{\alpha>m} \frac1{\alpha^2} e^{-2\alpha R^2|z_j-z_i|^2}\,|dz_j||dz_i|.
\end{align*}
\begin{rem}
  The identity \eqref{eq: exp wick prod} together with \eqref{eq: Wick log} essentially proves Lemma~\ref{lem: dilog}.
\end{rem}

Using Lemma~\ref{cl:abs int}, we have
\begin{align*}
  |\E\big[ (\Delta_i - \Delta_i^{(m)}) (\Delta_j - \Delta_j^{(m)})\big]| &\leq \frac14 \int_{\Gamma_i}\int_{\Gamma_j} \Abs{\frac{\partial^2}{\partial \hat{n}_i\partial \hat{n}_j} \sum_{\alpha>m} \frac1{\alpha^2} e^{-2\alpha R^2|z_j-z_i|^2}}\,|dz_j||dz_i|\\
  &\leq CR \sum_{\alpha>m} \frac1{\alpha^{3/2}} \leq C\frac{R}{\sqrt{m}}.
\end{align*}
Finally
\begin{equation*}
  \E\big[ ( \overline{\Delta}_R-\Delta_R^{(m)} )^2 \big] \leq \sum_{i,j=1}^N a_i a_j \Abs{\E\big[ (\Delta_i - \Delta_i^{(m)})(\Delta_j - \Delta_j^{(m)})\big]}\leq C\frac{R}{\sqrt{m}},
\end{equation*}
which is \eqref{eq: L2 approx for CLT}.

\subsection{CLT for the approximant}
We finish by proving \eqref{eq: CLT m}. We claim that it's enough to prove that for any non-negative integers $p_1,\dots,p_N$ we have, as $R\to\infty$,
\begin{equation}\label{eq: mix mom Delta m}
  \E\Big[\prod_{i=1}^N(\Delta_i^{(m)})^{p_i}\Big] = \E\Big[\prod_{i=1}^N \xi_i^{p_i}\Big] + o(R^{P/2})
\end{equation}
where $\xi_i$ is a sequence of jointly (real) Gaussian random variables, with mean $0$ and covariance
\begin{equation*}
  \cov (\xi_i,\xi_j)=\E[\xi_i\xi_j]=\cov(\Delta_i^{(m)},\Delta_j^{(m)}),
\end{equation*}
and $P=p_1+\cdots +p_N$.
\begin{rem}
  Notice that
  \begin{equation*}
    \var\Delta_i^{(m)} = \frac{\sqrt{\pi}} 2\left(\sum_{\alpha=1}^m\frac1{\alpha^{3/2}}\right) \cL(\Gamma_i,\Gamma_i) R (1+o(1))\quad\text{as }R\to\infty,
  \end{equation*}
  which follows from combining \eqref{eq: cov Delta m} and Lemma~\ref{cl:asym of int}. By hypothesis, $\Gamma$ is a non-zero $\R$-chain and so we may assume that $\cL(\Gamma_i,\Gamma_i)\neq0$ for each $i$, which means that
  \begin{equation*}
    \prod_{i=1}^N \big(\var\Delta_i^{(m)}\big)^{p_i/2}\simeq R^{P/2}.
  \end{equation*}
  Now, defining $\widehat{\xi}_i=\tfrac{\xi_i}{\sqrt{\var\xi_i}}=\tfrac{\xi_i}{\sqrt{\var\Delta_i}}$, we see that \eqref{eq: mix mom Delta m} is equivalent to
  \begin{equation*}
    \E\Bigg[\prod_{i=1}^N\left(\frac{\Delta_i^{(m)}}{\sqrt{\var\Delta_i^{(m)}}}\right)^{p_i}\Bigg] = \E\Big[\prod_{i=1}^N \widehat{\xi_i}^{p_i}\Big] +o(1),
  \end{equation*}
  and note that $\Abs{\E\Big[\prod_{j=1}^N \widehat{\xi_i}^{p_i}\Big]}$ is a bounded quantity.
\end{rem}

To see that it suffices to show \eqref{eq: mix mom Delta m}, notice that it implies that, for any non-negative integer $p$,
\begin{align*}
  \E[ (\Delta_R^{(m)})^{p} ] &= \sum_{i_1,\dots,i_p=1}^N a_{i_1}\dots a_{i_p} \E\Big[ \prod_{k=1}^p \Delta_{i_k}^{(m)} \Big]\\
  &\sim \sum_{i_1,\dots,i_p=1}^N a_{i_1}\dots a_{i_p} \E\Big[ \prod_{k=1}^p \xi_{i_k} \Big] = \E\Big[ \big(\sum_{i=1}^N a_{i} \xi_{i}\big)^p \Big].
\end{align*}
Now since $\sum_{i=1}^N a_{i} \xi_{i}$ is a mean $0$ real Gaussian with the same variance as $\Delta_R^{(m)}$, we see that the moments of $\frac{\Delta_R^{(m)}}{\sqrt{\var \Delta_R^{(m)}}}$ converge to the moments of the standard real Gaussian, which implies \eqref{eq: CLT m}.

It remains to establish \eqref{eq: mix mom Delta m}. We begin by re-formulating the right-hand side, and so we introduce some notation. Throughout this computation the integers $p_1,\dots,p_N$ and $m$ are considered to be fixed, and we often ignore the dependence of other parameters on them. We define $\widehat{p}_i=p_1+\dots+p_i$ for $1\leq i \leq N$, and note that $P=\widehat{p}_N$ which we will use interchangeably according to the context. We define a new sequence of random variables $(\widetilde{\xi}_r)_{1\leq r \leq P}$ by
\begin{align*}
  \widetilde{\xi}_1& =\widetilde{\xi}_2 = \dots = \widetilde{\xi}_{p_1} = \xi_1,\\
  \widetilde{\xi}_{p_1+1}& =\widetilde{\xi}_{p_1+2} = \dots = \widetilde{\xi}_{\widehat{p}_2} = \xi_2,\\
   \vdots  & \\
  \widetilde{\xi}_{\widehat{p}_{N-1}+1}& =\widetilde{\xi}_{\widehat{p}_{N-1}+2} = \dots = \widetilde{\xi}_{\widehat{p}_{N}} = \xi_N.
\end{align*}
A \emph{partition $\cP=\biguplus_k\{r_k,s_k\}$} is a partition of the set $\{1,\dots,P\}$ into pairs $\{r_k,s_k\}$. We always label the partition so that $r_k<s_k$ and $r_k<r_{k'}$ for $k<k'$. Of course if $P$ is odd then no such partition exists. Now \cite{Jan}*{Theorem 1.28} implies that
\begin{equation}\label{eq: RHS rewritten}
  \E\Big[\prod_{j=1}^N \xi_i^{p_i}\Big]=\sum_\cP\prod_k\E[\widetilde{\xi}_{r_k}\widetilde{\xi}_{s_k}].
\end{equation}
In particular this expectation is zero if $P$ is odd.

We now consider the left-hand side of \eqref{eq: mix mom Delta m}. Since
\begin{equation*}
  \Delta_i^{(m)} = \int_{\Gamma_i} \frac{\partial}{\partial \hat{n}_i} \left(-\frac\gamma2+\sum_{\alpha=1}^m \frac{c_\alpha}{\alpha!} \Wiaa{\widehat{f}_R(z)}{\alpha}\right)\,|dz_i| = \sum_{\alpha=1}^m \frac{c_\alpha}{\alpha!} \int_{\Gamma_i} \frac{\partial}{\partial \hat{n}_i} \Wiaa{\widehat{f}_R(z)}{\alpha} \,|dz_i|
\end{equation*}
we see that, denoting by $\Gamma_i^{p_i}$ the Cartesian product of $p_i$ copies of $\Gamma_i$,
\begin{equation*}
  \E\Big[\prod_{i=1}^N(\Delta_i^{(m)})^{p_i}\Big] = \sum_{\alpha_1,\dots,\alpha_{P}=1}^N \prod_{r=1}^{P} \frac{c_{\alpha_r}}{\alpha_r!} \E\Big[ \int_{\prod_{i=1}^{N}\Gamma_i^{p_i}} \frac{\partial^{P}}{\partial \hat{n}_1\dots \partial \hat{n}_{P}} \prod_{r=1}^{P} \Wiaa{\widehat{f}_R(z_r)}{\alpha_r} \,\prod_{r=1}^{P}|dz_r| \Big].
\end{equation*}
Lemma~\ref{lem: hypoth sat polys} allows us to apply Lemma~\ref{lem: abstract interchange} to see that
\begin{equation}\label{eq: B1}
  \E\Big[\prod_{i=1}^N(\Delta_i^{(m)})^{p_i}\Big] = \sum_{\alpha_1,\dots,\alpha_{P}=1}^m \prod_{r=1}^{P} \frac{c_{\alpha_r}}{\alpha_r!} \int_{\prod_{i=1}^{N}\Gamma_i^{p_i}} \frac{\partial^{P}}{\partial \hat{n}_1\dots \partial \hat{n}_{P}} \E\Big[ \prod_{r=1}^{P} \Wiaa{\widehat{f}_R(z_r)}{\alpha_r}\Big] \,\prod_{r=1}^{P}|dz_r|.
\end{equation}
We will compute the asymptotics of this expression via the diagram formula. For $1\leq i,j\leq N$ and $1\leq\alpha\leq m$ we write
\begin{equation*}
  \rho_{ij}(\alpha)=\int_{\Gamma_i}\int_{\Gamma_j} \frac{\partial^{2}}{\partial \hat{n}_i \partial \hat{n}_{j}} \Abs{\widehat{K}_R(z_i,z_j)}^{2\alpha} \,|dz_j||dz_i|.
\end{equation*}
We then have, from \eqref{eq: exp wick prod},
\begin{equation}\label{eq: m cov ij}
  \cov(\Delta_i^{(m)},\Delta_j^{(m)}) = \sum_{\alpha=1}^m c_{\alpha}^2 \rho_{ij}(\alpha).
\end{equation}

\subsubsection{Diagrams}
Given non-negative integers $\alpha_1,\dots,\alpha_P$, a \emph{diagram $\cD$} is a graph with $2(\alpha_1+\dots+\alpha_P)$ vertices such that:
\begin{itemize}
  \item For each $1\leq r \leq P$ there are $\alpha_r$ vertices labelled $r$ and $\alpha_r$ vertices labelled $\bar{r}$.
  \item Each vertex has degree exactly $1$.
  \item Each edge joins a vertex labelled $r$ to a vertex labelled $\bar{s}$ for $r\neq s$.
\end{itemize}
Note that there are choices of $\alpha_1,\dots,\alpha_P$ such that no such diagram exists, for example if $\alpha_1>\alpha_2+\cdots+\alpha_P$. We denote the edges (respectively the vertices) of $\cD$ by $e(\cD)$ (respectively $v(\cD)$).

Recall that
\begin{equation*}
\widehat{K}_R(z_r,z_s)=\frac{K_R(z_r,z_s)}{\sqrt{K_R(z_r,z_s)K_R(z_r,z_s)}} = \exp\{R^2(z_r\bar{z}_s -\tfrac12|z_r|^2 -\tfrac12|z_s|^2)\}.
\end{equation*}
The \emph{value} of a diagram is
\begin{equation*}
  V(\cD)=\prod_{(r,\bar{s})\in e(\cD)} \widehat{K}_R(z_r,z_s).
\end{equation*}
The diagram formula \cite{Jan}*{Theorem 3.12} implies that
\begin{equation}\label{eq: diagram formula}
  \E\Big[ \prod_{r=1}^{P} \Wiaa{\widehat{f}_R(z_r)}{\alpha_r}\Big]=\sum_\cD V(\cD).
\end{equation}

We say that a diagram is \emph{regular} if the set $\{1,\dots,P\}$ can be partitioned into pairs $\{r_k,s_k\}$ such that each edge of the diagram is of the form $(r_k,\overline{s}_k)$ or $(s_k,\overline{r}_k)$ for some $k$; otherwise the diagram is said to be irregular, see Figure~\ref{fig: diags}. Note that if $P$ is odd then all diagrams are irregular. We again label the partition so that $r_k<s_k$ and $r_k<r_{k'}$ for $k<k'$.
\begin{figure}
  \centering
  \includegraphics[width=0.5\columnwidth]{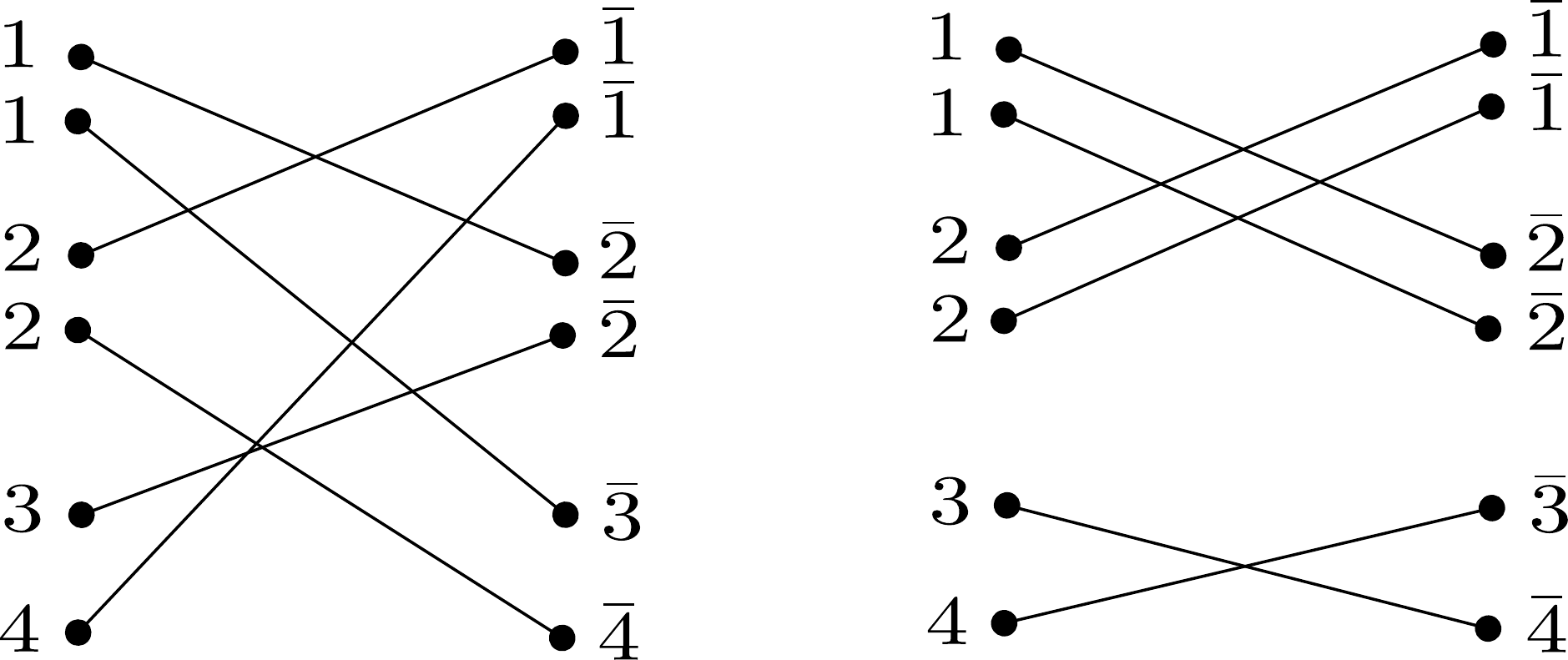}
  \caption{An irregular and a regular diagram for $P=4$, $\alpha_1=\alpha_2=2$ and $\alpha_3=\alpha_4=1$.}
  \label{fig: diags}
\end{figure}

Combining \eqref{eq: B1} and \eqref{eq: diagram formula} we have
\begin{equation}\label{eq: full moment expansion}
  \E\Big[\prod_{i=1}^N(\Delta_i^{(m)})^{p_i}\Big] = \sum_{\alpha_1,\dots,\alpha_{P}=1}^m \prod_{r=1}^{P} \frac{c_{\alpha_r}}{\alpha_r!} \sum_\cD \int_{\prod_{r=1}^{P}\Gamma_i^{p_i}} \frac{\partial^{P}}{\partial \hat{n}_1\dots \partial \hat{n}_{P}} V(\cD) \,\prod_{r=1}^{P}|dz_r|.
\end{equation}
We now split this sum into two pieces, by splitting $\displaystyle{\sum_\cD = \sum_{\mathrm{regular}\,\cD} + \sum_{\mathrm{irregular}\,\cD}}$. We estimate each contribution separately - we shall see that the regular contribution will give us the main term on the right-hand side of \eqref{eq: mix mom Delta m} while the irregular contribution will give the error term. We begin by computing the regular part exactly.

\paragraph{{\bf The regular contribution}}\hspace{0pt} \\
We define the \emph{multiplicity vector} $\overrightarrow{B}=(\beta_1,\dots,\beta_{Q})$ of a regular diagram by $\beta_k=\alpha_{r_k}=\alpha_{s_k}$; here $P=2Q$. Notice that for a regular diagram
\begin{equation*}
  V(\cD)=\prod_k |\widehat{K}_R(z_{r_k},z_{s_k})|^{2\beta_k}.
\end{equation*}
Given a regular diagram $\cD$ with partition $\cP$ and multiplicity vector $\overrightarrow{B}$, we have
\begin{align*}
  \int_{\prod_{r=1}^{P}\Gamma_i^{p_i}} \frac{\partial^{P}}{\partial \hat{n}_1\dots \partial \hat{n}_{P}} V(\cD) \,\prod_{r=1}^{P}|dz_r| &= \int_{\prod_{r=1}^{P}\Gamma_i^{p_i}} \frac{\partial^{P}}{\partial \hat{n}_1\dots \partial \hat{n}_{P}} \prod_{k=1}^Q |\widehat{K}_R(z_{r_k},z_{s_k})|^{2\beta_k} \,\prod_{r=1}^{P}|dz_r|\\
  &= \prod_{k=1}^Q \int_{\Gamma_{r_k}}\int_{\Gamma_{s_k}} \frac{\partial^{2}}{\partial \hat{n}_{r_k} \partial \hat{n}_{s_k}} |\widehat{K}_R(z_{r_k},z_{s_k})|^{2\beta_k} \,|dz_{s_k}||dz_{r_k}|\\
  &= \prod_{k=1}^Q \rho_{r_k,s_k}(\beta_k).
\end{align*}
We now need to count the number of regular diagrams with partition $\cP$ and multiplicity vector $\overrightarrow{B}$. The ordering of the partition we specified, combined with the multiplicity vector $\overrightarrow{B}$ uniquely defines the values $\alpha_1,\dots,\alpha_P$. Given these values we may permute the $\alpha_r$ vertices labelled $r$, independently for each $r$, to get all of the regular diagrams corresponding to these values, $\cP$ and $\overrightarrow{B}$. There are $\prod_{r=1}^P \alpha_r!$ such permutations. Noting that $\prod_{r=1}^P c_{\alpha_r}= \prod_{k=1}^Q c_{\beta_k}^2$, we get that the regular contribution is
\begin{align*}
  \sum_{\alpha_1,\dots,\alpha_{P}=1}^m \prod_{r=1}^{P} \frac{c_{\alpha_r}}{\alpha_r!} \sum_{\mathrm{regular}\,\cD} \prod_{k=1}^Q \rho_{r_k,s_k}(\beta_k) &= \sum_{\beta_1,\dots,\beta_{Q}=1}^m \prod_{k=1}^{Q} c_{\beta_k}^2 \sum_{\cP} \prod_{k=1}^Q \rho_{r_k,s_k}(\beta_k)\\
  &= \sum_{\cP} \sum_{\beta_1,\dots,\beta_{Q}=1}^m \prod_{k=1}^{Q} c_{\beta_k}^2 \rho_{r_k,s_k}(\beta_k)\\
  &= \sum_{\cP} \prod_{k=1}^{Q} \left(\sum_{\beta=1}^m  c_{\beta}^2 \rho_{r_k s_k} (\beta) \right)\\
  &\overset{\eqref{eq: m cov ij}}{=} \sum_{\cP} \prod_{k=1}^{Q} \cov(\Delta_{r_k}^{(m)},\Delta_{s_k}^{(m)})\\
  &\overset{\eqref{eq: RHS rewritten}}{=} \E\Big[\prod_{j=1}^N \xi_i^{p_i}\Big].
\end{align*}

\paragraph{{\bf The irregular contribution}}\hspace{0pt} \\
It remains to see only that the irregular contribution is $o(R^{P/2})$. Further, from \eqref{eq: full moment expansion}, we see that it is enough to bound
\begin{equation*}
  \int_{\prod_{i=1}^{N}\Gamma_i^{p_i}} \frac{\partial^{P}}{\partial \hat{n}_1\dots \partial \hat{n}_{P}} V(\cD) \,\prod_{r=1}^{P}|dz_r|
\end{equation*}
for each irregular diagram $\cD$. Recalling that
\begin{equation*}
  V(\cD)= \prod_{(r,\bar{s})\in e(\cD)} \widehat{K}_R(z_r,z_s) = \prod_{(r,\bar{s})\in e(\cD)} \exp\{R^2(z_r\bar{z}_s -\tfrac12|z_r|^2 -\tfrac12|z_s|^2)\}
\end{equation*}
we have
\begin{equation*}
  \log V(\cD)= R^2 \sum_{(r,\bar{s})\in e(\cD)} (z_r\bar{z}_s -\tfrac12|z_r|^2 -\tfrac12|z_s|^2).
\end{equation*}
Now there are $\alpha_r$ edges in the sum of the form $(r,\bar{s})$ for some $s$, and $\alpha_r$ edges of the form $(s,\bar{r})$ for some $s$. We therefore have
\begin{equation*}
  \log V(\cD)= R^2 \sum_{(r,\bar{s})\in e(\cD)} (z_r(\bar{z}_s -\bar{z}_r)) = R^2 \sum_{(r,\bar{s})\in e(\cD)} (\bar{z}_s(z_r -z_s))
\end{equation*}
from which we conclude that, for $1\leq t \leq P$,
\begin{equation*}
  \frac{\partial}{\partial z_t}V(\cD) = R^2 V(\cD)\sum_{s:(t,\bar{s})\in e(\cD)} (\bar{z}_s -\bar{z}_t)
\end{equation*}
and
\begin{equation*}
  \frac{\partial}{\partial \bar{z}_t}V(\cD) = R^2 V(\cD)\sum_{r:(r,\bar{t})\in e(\cD)} (z_r - z_t).
\end{equation*}
By iterating this argument we see that we may bound
\begin{equation*}
  \Abs{\frac{\partial^{P}}{\partial \hat{n}_1\dots \partial \hat{n}_{P}} V(\cD)}
\end{equation*}
by a (finite) linear combination of terms of the form
\begin{equation*}
  \Abs{V(\cD)} R^{2(P-\gamma)}\fp_{P-2\gamma},
\end{equation*}
where $0\leq\gamma\leq\lfloor\tfrac P2\rfloor$ is an integer and $\fp_{P-2\gamma}$ is a product of $P-2\gamma$ factors (not necessarily distinct) of the form $|z_r-z_s|$ with $(r,\bar{s})\in e(\cD)$. To finish the proof it therefore suffices to see that
\begin{equation*}
  R^{2(P-\gamma)} \int_{\prod_{i=1}^{N}\Gamma_i^{p_i}} \Abs{V(\cD)} \fp_{P-2\gamma} \,\prod_{r=1}^{P}|dz_r| = o(R^{P/2})
\end{equation*}
for any choice of $\gamma$ and $\fp_{P-2\gamma}$.

We now fix $\gamma$ and $\fp_{P-2\gamma}$ and make a reduction to allow us to estimate this quantity. From the irregular diagram $\cD$ we form the reduced diagram $\cD^*$ (see Figure~\ref{fig: red diag}) with $P$ vertices (labelled $1$ to $P$) such that:
\begin{itemize}
  \item For each $1\leq r,s \leq P$ there is at most one edge $(r,s)$.
  \item $(r,s)\in e(\cD^*)$ if $(r,\bar{s})\in e(\cD)$ or $(s,\bar{r})\in e(\cD)$.
\end{itemize}
\begin{figure}
  \centering
  \includegraphics[width=0.5\columnwidth]{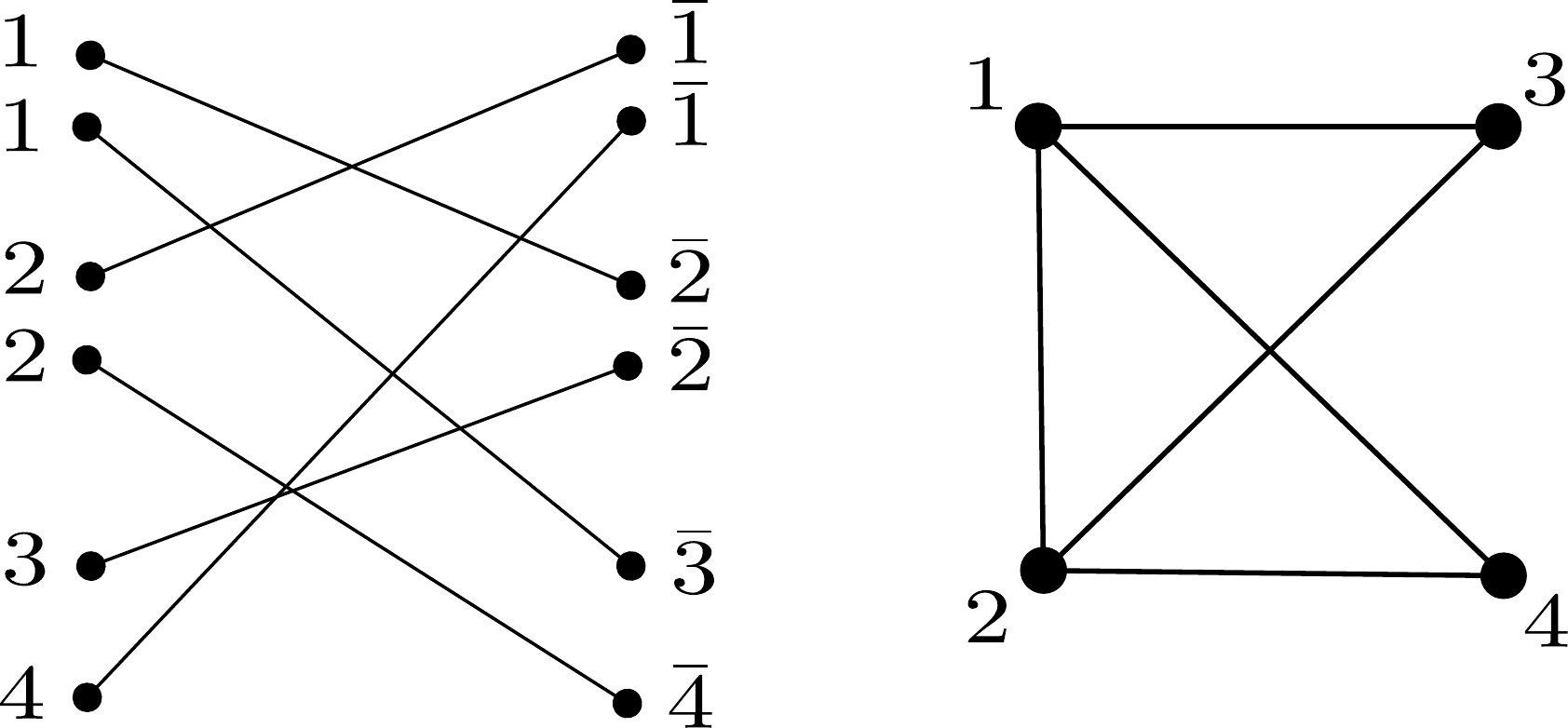}
  \caption{A diagram and its reduced diagram.}
  \label{fig: red diag}
\end{figure}
In other words we form $\cD^*$ from $\cD$ by glueing together the $2\alpha_r$ vertices labelled $r$ or $\bar{r}$ for each $r$, and ignoring the multiplicity of the edges of the resultant diagram. We decompose
\begin{equation*}
  \cD^*=\bigcup_{u=1}^n \cD_u
\end{equation*}
into $n$ connected components that contain $a_u$ vertices and contribute $\ell_u$ factors to $\fp_{P-2\gamma}$. Notice that $n<\tfrac P2$ since $\cD$ is irregular, and that
\begin{equation*}
  \sum_{u=1}^n a_u=P\qquad\text{ and }\qquad  \sum_{u=1}^n \ell_u = P-2\gamma.
\end{equation*}
Moreover, since
\begin{equation*}
|V(\cD)|=\prod_{(r,\bar{s})\in e(\cD)} \exp\{-\tfrac{R^2}2 |z_r-z_s|^2 \} \leq \prod_{(r,s)\in e(\cD^*)} \exp\{ -\tfrac{R^2}2 |z_r-z_s|^2 \}
\end{equation*}
we may factorise the expression we seek to bound as
\begin{align*}
  R^{2(P-\gamma)} \int_{\prod_{i=1}^{N}\Gamma_i^{p_i}} \Abs{V(\cD)} &\fp_{P-2\gamma} \,\prod_{r=1}^{P}|dz_r|\\
  &\leq R^{2(P-\gamma)} \prod_{u=1}^n \int_{\prod \widetilde{\Gamma}_r} \prod_{(r,s)\in e(\cD_u)}e^{-\tfrac{R^2}2|z_r-z_s|^2 } \fp_{\ell_u} \,\prod_{r\in v(\cD_u)}|dz_r|
\end{align*}
where, $\widetilde{\Gamma}_r = \Gamma_i$ for some $i$, $\prod \widetilde{\Gamma}_r = \prod_{r\in v(\cD_u)}\widetilde{\Gamma}_r$ and $\fp_{\ell_u}$ means a product of $\ell_u$ factors of the form $|z_r-z_s|$ with $(r,s)\in e(\cD_u)$. We will show that
\begin{equation}\label{eq: bound irr diag}
  \int_{\prod \widetilde{\Gamma}_r} \prod_{(r,s)\in e(\cD_u)} e^{- \tfrac{R^2}2|z_r-z_s|^2 } \fp_{\ell_u} \,\prod_{r\in v(\cD_u)}|dz_r|\lesssim (\log R)^{\ell_u /2}R^{-(\ell_u + a_u -1)}
\end{equation}
which will yield
\begin{align*}
  R^{2(P-\gamma)} \int_{\prod_{i=1}^{N}\Gamma_i^{p_i}} \Abs{V(\cD)} \fp_{P-2\gamma} \,\prod_{r=1}^{P}|dz_r| &\lesssim R^{2(P-\gamma)} (\log R)^{\sum_u\ell_u /2} R^{-\sum_u(\ell_u + a_u -1)}\\
  &=(\log R)^{\tfrac P2 -\gamma} R^{2P-2\gamma-(P-2\gamma + P -n)}\\
  &= (\log R)^{\tfrac P2 -\gamma} R^{n}=o(R^{P/2})
\end{align*}
as claimed, since $n<\tfrac P2$.

\paragraph{{\bf Proof of estimate \eqref{eq: bound irr diag}}}\hspace{0pt} \\
It remains only to prove \eqref{eq: bound irr diag}. We formulate it as follows: Let $G$ be a connected graph with $a$ vertices, let $\{\widetilde{\Gamma}_1,\dots,\widetilde{\Gamma}_a\} \subset \{\Gamma_1,\dots,\Gamma_N\}$ be a collection of curves (we allow repition) and let $\fp_{\ell}$ be a product of $\ell$ factors of the form $|z_r-z_s|$ with $1\leq r,s\leq a$. Then
\begin{equation}\label{eq: simple bound irr diag}
  \int_{\prod \widetilde{\Gamma}_r} \prod_{(r,s)\in e(G)} e^{- \tfrac{R^2}2|z_r-z_s|^2 } \fp_{\ell} \,\prod_{r=1}^a|dz_r|\lesssim (\log R)^{\ell /2}R^{-(\ell + a -1)}.
\end{equation}
First note that since $e^{- \tfrac{R^2}2|z_r-z_s|^2 }\leq 1$ we may delete some of the edges of $G$ to form a tree. By re-labelling the vertices we may assume that deleting the vertices labelled $1,\dots,r$ yields a connected a graph, for every $r$. We denote by $s(r+1)$ the vertex that is joined to $r+1$ in this reduced graph. See Figure~\ref{fig: tree}. Note that
\begin{equation*}
  \prod_{(r,s)\in e(G)} e^{- \tfrac{R^2}2|z_r-z_s|^2 }\leq \prod_{r=1}^{a-1} e^{- \tfrac{R^2}2|z_r-z_{s(r)}|^2 }.
\end{equation*}
\begin{figure}
  \centering
  \includegraphics[width=0.5\columnwidth]{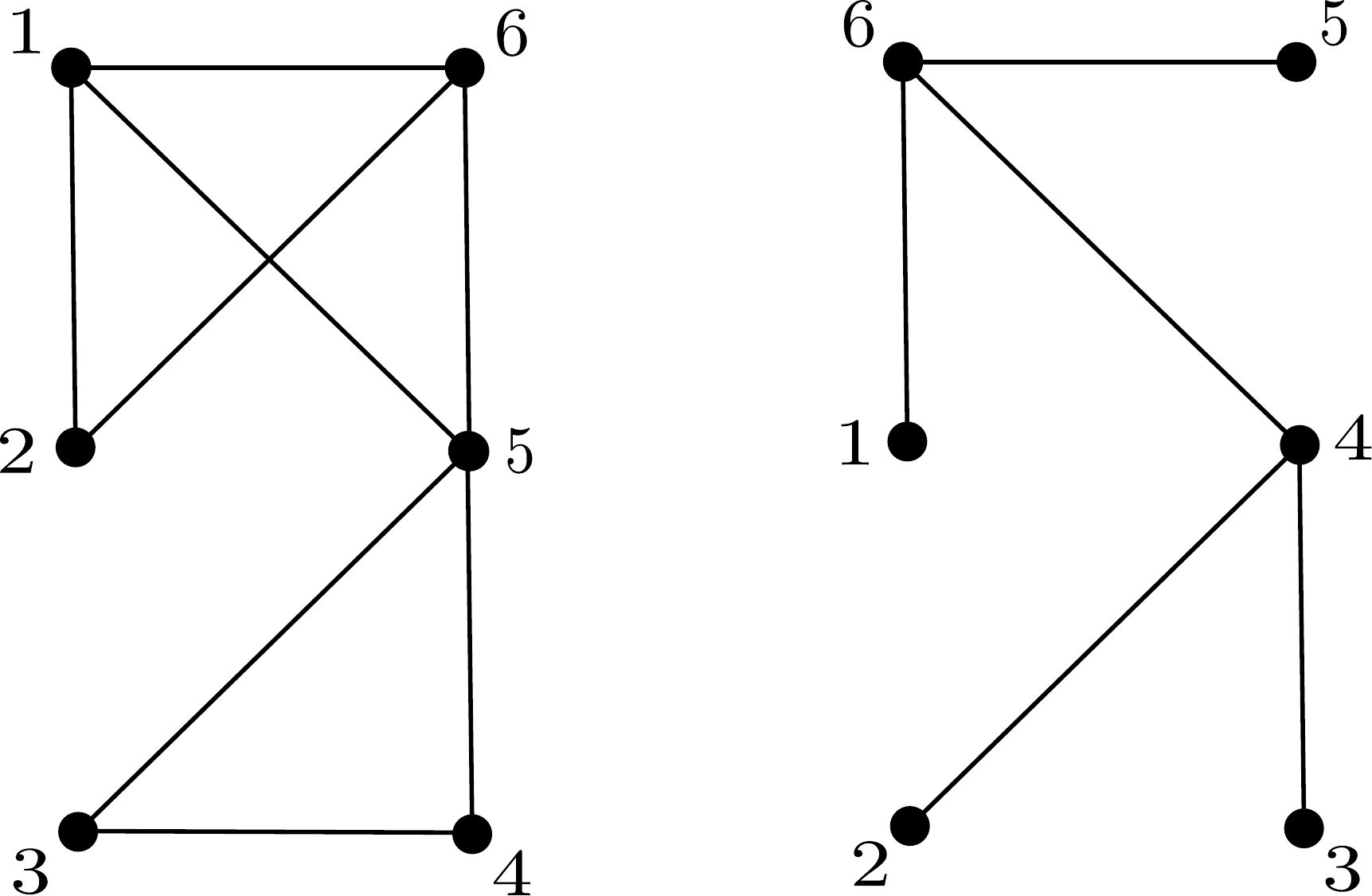}
  \caption{A connected graph and a tree formed by deleting some edges. The vertices are also re-labelled so that successively deleting the vertices labelled $1,2,\dots$ yields a connected graph at every step, $s(1)=s(4)=s(5)=6$ and $s(2)=s(3)=4$.}
  \label{fig: tree}
\end{figure}

Define $\eps_R=\frac{\sqrt{(\ell+a-1)\log R}}R $ and note that if $|z_r-z_{s(r)}|>\eps_R$ for some $r$ then
\begin{equation*}
  \prod_{r=1}^{a-1} e^{- \tfrac{R^2}2|z_r-z_{s(r)}|^2 }\leq R^{-(\ell + a -1)}.
\end{equation*}
Since $\fp_{\ell}$ is uniformly bounded on $\prod \widetilde{\Gamma}_r$, and the curves have finite length, to show \eqref{eq: simple bound irr diag} it suffices to bound
\begin{equation*}
  \int_{\widetilde{\Gamma}_a\times\prod_{r=1}^{a-1} C_r} \prod_{r=1}^{a-1} e^{- \tfrac{R^2}2|z_r-z_{s(r)}|^2 } \fp_{\ell} \,\prod_{r=1}^a|dz_r|,
\end{equation*}
where $C_r=\{z_r\in\widetilde{\Gamma}_r\colon |z_r-z_{s(r)}|\leq\eps_R \}$. Note that in this new domain of integration we have $|z_r-z_s|\leq a\eps_R$, which implies that $\fp_{\ell} \lesssim \eps_R^\ell \lesssim \left( \frac{\sqrt{\log R}} R \right)^\ell$. It therefore suffices to see that
\begin{equation*}
  \int_{\widetilde{\Gamma}_a\times\prod_{r=1}^{a-1} C_r} \prod_{r=1}^{a-1} e^{- \tfrac{R^2}2|z_r-z_{s(r)}|^2 } \,\prod_{r=1}^a|dz_r|\lesssim R^{1-a}.
\end{equation*}
We claim that
\begin{equation}\label{eq: irr estimate one int}
  \int_{C_1} e^{- \tfrac{R^2}2|z_1-z_{s(1)}|^2 } \,|dz_1|\lesssim R^{-1},
\end{equation}
uniformly in the remaining variables. Applying this estimate $a-1$ times (replacing the index $1$ by $2,\dots,a-1$) yields
\begin{equation*}
  \int_{\widetilde{\Gamma}_a\times\prod_{r=1}^{a-1} C_r} \prod_{r=1}^{a-1} e^{- \tfrac{R^2}2|z_r-z_{s(r)}|^2 } \,\prod_{r=1}^a|dz_r|\lesssim R^{1-a}\int_{\widetilde{\Gamma}_a}|dz_a|\lesssim R^{1-a}.
\end{equation*}

\paragraph{{\bf Proof of estimate \eqref{eq: irr estimate one int}}}\hspace{0pt} \\
It remains to show \eqref{eq: irr estimate one int}. Fix $(z_2,\dots,z_a)\in \prod_{r=2}^{a-1} C_r\times\widetilde{\Gamma}_a$ and define $z_1^*$ to be the point in $\widetilde{\Gamma}_1$ closest to $z_{s(1)}$. (If there are many points we choose one arbitrarily; it might be the case that $z_1^*=z_{s(1)}$.) We have $|z_1-z_{s(1)}|\geq\tfrac12|z_1-z_1^*|$ since $|z_1-z_{s(1)}|\geq|z_1-z_1^*|-|z_1^*-z_{s(1)}|\geq|z_1-z_1^*|-|z_1-z_{s(1)}|$. This yields
\begin{equation*}
  \int_{C_1} e^{- \tfrac{R^2}2|z_1-z_{s(1)}|^2 } \,|dz_1|\leq\int_{C_1} e^{- \tfrac{R^2}8|z_1-z_1^*|^2 } \,|dz_1|,
\end{equation*}
which we bound exactly as in Section~\ref{sec: proof of claims}. Let $\gamma\colon[0,1]\to\C$ be a parameterisation of $\widetilde{\Gamma}_1$ satisfying $M_1\leq|\dot{\gamma}(t)|\leq M_2$ and $M_1'|t-s|\leq|\gamma_j(t)-\gamma_j(s)|\leq M_2|t-s|$ with $M_1,M_1'>0$. Denote by $t^*$ the (unique) value such that $\gamma(t^*)=z_1^*$. Notice that
\begin{equation*}
  C_1\subseteq\left\{z_1=\gamma(t)\colon |t-t^*|\leq\frac{\eps_R}{M_1'}\right\}.
\end{equation*}
We then bound
\begin{align*}
  \int_{C_1} e^{- \tfrac{R^2}8|z_1-z_1^*|^2 } \,|dz_1|&\leq \int_{|t-t^*|\leq\tfrac{\eps_R}{M_1'}} e^{- \tfrac{R^2}8|\gamma(t)-\gamma(t^*)|^2 } |\dot{\gamma}(t)|\,dt\\
  &\leq M_2 \int_{|t-t^*|\leq\tfrac{\eps_R}{M_1'}} e^{- \tfrac{R^2}8 (M_1')2(t-t^*)^2 } \,dt\\
  &=\frac{2\sqrt2M_2}{RM_1'}\int_{|s|\leq\tfrac{R\eps_R}{2\sqrt2}} e^{- s^2 } \,ds\\
  &\lesssim R^{-1}.
\end{align*}
This completes the proof.

\end{document}